\documentclass[11pt]{article}

\usepackage[a4paper,margin=1in]{geometry}
\usepackage{amsmath,amssymb,amsthm,mathtools}
\usepackage{bm}
\usepackage{enumitem}
\usepackage{microtype}
\usepackage{hyperref}
\usepackage{cleveref}
\usepackage{comment}

\hypersetup{
  colorlinks=true,
  linkcolor=blue,
  citecolor=blue,
  urlcolor=blue
}

\usepackage{xcolor}      % explicit:  revision marks
\usepackage{graphicx}
\usepackage{subcaption}
\graphicspath{{figs/}}

\theoremstyle{plain}
\newtheorem{theorem}{Theorem}[section]
\newtheorem{proposition}[theorem]{Proposition}
\newtheorem{lemma}[theorem]{Lemma}
\newtheorem{corollary}[theorem]{Corollary}

\theoremstyle{definition}
\newtheorem{assumption}[theorem]{Assumption}

\newtheorem{remark}[theorem]{Remark}

\crefname{section}{Section}{Sections}
\Crefname{section}{Section}{Sections}
\DeclareMathOperator*{\argmin}{arg\,min}

\newcommand{\Rd}{\mathbb{R}^d}
\newcommand{\Lip}{\mathrm{Lip}}
\newcommand{\norm}[1]{\left\lVert #1 \right\rVert}

\title{Policy Iteration for Stationary Discounted Hamilton--Jacobi--Bellman Equations: A Viscosity Approach}
\author{
Namkyeong Cho\thanks{Department of Financial Mathematics, Gachon University, Korea. \texttt{namkyeong.cho@gmail.com}}
\and
Yeoneung Kim\thanks{Department of Industrial Engineering, Yonsei University, Republic of Korea. \texttt{yeoneung@yonsei.ac.kr}}
}
\date{}

\begin{document}
\maketitle

\begin{abstract}
{
We study policy iteration (PI) for deterministic infinite-horizon discounted
control problems characterized by stationary Hamilton--Jacobi--Bellman
equations. For general viscosity solutions, the classical gradient-based
policy improvement step need not be defined pointwise. We introduce a
semi-discrete formulation with centered difference quotients at scale $h$
and a separate artificial-viscosity term of order $O(h)$. The resulting
stencil is monotone, and the positive discount yields a resolvent contraction.
Under bounded Lipschitz data and a globally Lipschitz minimizing policy map,
we prove monotone and geometric convergence of the value iterates for each
fixed $h>0$, together with a local quadratic estimate whose constant is of
order $h^{-2}$. Under the additional condition $\lambda>\Lip_x(f)$, we establish
$\|V^h-V\|_\infty\le C\sqrt h$ and combine the discretization and iteration
errors into a quantitative bound. A bounded Lipschitz example shows that the
$\sqrt h$ exponent is sharp for this scheme. The combined estimate gives a
sufficient iteration count of order $h^{-1}\log(1/h)$ to attain an error of
order $\sqrt h$. In bounded-domain experiments, the smooth one-dimensional
benchmark exhibits the predicted discretization plateau, while a nonlinear
two-dimensional manufactured benchmark isolates convergence to the discrete
solution. Exact policy evaluation gives substantially faster local
convergence than the global geometric bound. A neural evaluation diagnostic
illustrates the importance of controlling boundary errors as well as interior
residuals.}
\end{abstract}

% ============================================================

\section{Introduction}

Policy iteration (PI), originally introduced by Howard~\cite{Howard1960}, is a cornerstone of dynamic programming for solving optimal control and Markov decision processes. In discrete settings, PI enjoys strong structural properties, including monotonicity and geometric convergence \cite{Puterman1994,SantosRust2004}, and forms the basis of many reinforcement learning algorithms \cite{SuttonBarto2018}. 

In continuous time and space, optimal control problems are characterized by Hamilton--Jacobi--Bellman (HJB) equations, and PI can be formally interpreted as a nonlinear fixed-point method for such PDEs. The assumptions needed for a rigorous PDE-level analysis depend on the regularity and structure of the control problem. Existing results typically rely on special structures, such as linear--quadratic models \cite{Kleinman1968,Vrabie2009}, or stochastic control problems where diffusion provides regularization \cite{Kerimkulov2020,Puterman1981}. More recently, convergence results have been obtained for entropy-regularized and exploratory stochastic control problems under ellipticity assumptions \cite{HuangWangZhou2025,TranWangZhang2025}, where second-order elliptic structure plays a crucial role.

{For deterministic infinite-horizon problems with control constraints,
Kundu and Kunisch~\cite{KunduKunisch2024} analyze policy iteration under
classical differentiability and stabilization assumptions. Their deterministic
theory concerns an undiscounted stabilization problem and assumes a $C^1$ value
function and an admissible stabilizing initial policy. The present analysis
addresses bounded discounted problems in a viscosity setting that allows
nonsmooth value functions.}

Deterministic continuous-time control poses a regularity challenge. Even when the value function is globally Lipschitz continuous, its gradient $\nabla V$ may fail to exist pointwise. Global Lipschitz regularity itself requires additional control of the drift relative to the discount, as in Lemma~\ref{lem:V_lipschitz}. As a consequence, the classical policy improvement step
\[
\alpha_{n+1}(x) = \alpha(x,\nabla V_n(x))
\]
need not be defined pointwise. A formulation that accommodates nonsmooth values is therefore needed to analyze this update under general viscosity assumptions. A recent work~\cite{TangTranZhang2025} resolves this issue for deterministic finite-horizon problems. Their key idea is to introduce a monotone semi-discrete approximation by adding a viscosity term through finite differences in space. This artificial diffusion restores comparison and allows policy improvement to be performed using discrete gradients. Within this framework, policy iteration becomes well posed, and exponential convergence of the iterates can be established. Moreover, the approximation error is shown to be of order $\sqrt{h}$, consistent with classical viscosity approximation theory \cite{CrandallLions1984,BarlesSouganidis1991}.

{
We study the stationary discounted counterpart of this semi-discrete viscosity
framework, with quantitative bounds that combine policy iteration and
discretization errors. The stationary HJB equation
\[
\lambda V(x)+H(x,\nabla V(x))=0
\]
has no time variable. Its positive zeroth-order term permits a resolvent
contraction estimate, whereas the finite-horizon analysis in
\cite{TangTranZhang2025} uses time evolution. We make this stationary
contraction explicit and track its dependence on the discount and mesh size.}

The goal of this paper is to develop a rigorous viscosity-based policy iteration framework for deterministic infinite-horizon discounted control problems. Building on the semi-discrete approach of \cite{TangTranZhang2025}, {we replace the spatial gradient by centered difference quotients at scale $h$ and add an artificial-viscosity term of order $O(h)$ to ensure comparison at the discrete level.} This allows policy iteration to be formulated as a well-defined nonlinear fixed-point procedure.

Our contributions are threefold. First, for each fixed mesh size $h>0$, we establish monotone and geometric convergence of the policy iteration sequence toward the semi-discrete solution. The contraction arises from the resolvent structure of the discounted operator, rather than from time evolution. The same assumptions also give a local quadratic convergence estimate for fixed $h$, whose constant is not uniform in the mesh size. Second, under the additional condition $\lambda>\Lip_x(f)$, we prove {the canonical} vanishing-viscosity estimate
\[
\|V^h - V\|_{L^\infty(\mathbb{R}^d)} \lesssim \sqrt{h},
\]
which has the classical $\sqrt h$ order for Lipschitz solutions
\cite{CrandallLions1984}; Remark~\ref{rem:sqrt_h_sharpness} shows sharpness
within the present bounded-data class. Third, we derive a quantitative error decomposition that separates policy iteration error from discretization error and reveals a nontrivial coupling between the iteration count and the mesh size.

From a broader perspective, our work provides a PDE-based foundation for policy iteration in deterministic control. It complements recent developments in stochastic control, exploratory reinforcement learning, and entropy-regularized HJB equations \cite{HuangWangZhou2025,TranWangZhang2025}, as well as modern computational approaches based on operator learning, neural policy iteration, and physics-informed PDE solvers~\cite{lee2025hamilton,kim2025neural,kim2026physics,raissi2019physics,han2018solving}. In particular, these recent works demonstrate that policy-iteration-type ideas are increasingly important beyond classical control theory, but their numerical success still relies on structural ingredients such as regularization, stable policy evaluation, and well-posed policy improvement. Our analysis clarifies the role of monotonicity, viscosity regularization, and resolvent contraction in ensuring such stability and convergence in the deterministic stationary setting.

The remainder of the paper is organized as follows.
Section~\ref{sec:setup} introduces the discounted control problem
and explains the obstruction to classical pointwise policy improvement.
Section~\ref{sec:scheme} presents the semi-discrete scheme and the
associated PI algorithm.
Sections~\ref{sec:tools} and \ref{sec:wellposed}
establish the structural properties and well-posedness of the scheme.
Geometric convergence of policy iteration is proved in
Section~\ref{sec:expconv},
while the discretization error as $h\to0$ is analyzed in
Section~\ref{sec:disc_error}. Section~\ref{sec:exp} provides numerical validation.

\section{Problem setup}\label{sec:setup}
\subsection{Continuous-time discounted control}
Let $A\subset\mathbb{R}^m$ be nonempty and compact. Consider the controlled ODE
\begin{equation}\label{eq:ode}
\dot x(t)= f(x(t),a(t)),\qquad a(t)\in A,\qquad x(0)=x\in\Rd,
\end{equation}
with running cost $c(x,a)$ and discount factor $\lambda>0$. For a measurable control $a(\cdot)$,
define the discounted cost
\begin{equation}\label{eq:Jinf}
J(x;a):=\int_0^\infty e^{-\lambda t}\, c(x(t),a(t))\,dt.
\end{equation}
The value function is
\begin{equation}\label{eq:Vstar}
V(x):=\inf_{a(\cdot)}J(x;a).
\end{equation}
The associated Hamiltonian is defined by
\begin{equation}\label{eq:H}
H(x,p):=\sup_{a\in A}\{- c(x,a) -f(x,a)\cdot p \}.
\end{equation}

Throughout the paper, we work under the following assumptions.

\begin{assumption}\label{ass:A}
Assumptions on $f$, $c$ and $A$.
\begin{enumerate}[label=(A\arabic*),leftmargin=18pt]

\item\label{A1}
\textbf{Uniform boundedness and Lipschitz continuity.}
The functions $f(\cdot,\cdot)$ and $c(\cdot,\cdot)$ are uniformly bounded
and Lipschitz continuous in $(x,a)$:
\[
\norm{f}_{L^\infty(\Rd\times A)}+\norm{c}_{L^\infty(\Rd\times A)}<\infty,
\qquad
\Lip_x(f)+\Lip_x(c)+\Lip_a(f)+\Lip_a(c)<\infty.
\]

\item\label{A2}
\textbf{Compact control set.}
The control set $A\subset\mathbb R^m$ is nonempty and compact.

\item\label{A3}
\textbf{Discount factor.}
The discount parameter satisfies $\lambda>0$.

\end{enumerate}
\end{assumption}

Under Assumption~\ref{ass:A}, it is well known that the value function $V$
is the unique bounded viscosity solution of the stationary discounted HJB equation
\eqref{eq:HJB}; see, for example, \cite{CrandallIshiiLions1992,FlemingSoner2006,BardiCapuzzo1997,tran2021hamilton}.

\begin{equation}\label{eq:HJB}
\lambda V(x)+H(x,\nabla V(x))=0\qquad\text{in }\Rd.
\end{equation}

The presence of the zeroth-order term $\lambda V$ induces a resolvent structure in the stationary equation, which plays a stabilizing role analogous to time evolution in finite-horizon problems~\cite{TangTranZhang2025}. In particular, estimates deteriorate as $\lambda \downarrow 0$, reflecting the loss of coercivity in the discounted operator.

\subsection{Notation and semi-discrete operators}
\paragraph{Basic notation.}
Throughout the paper, $d\ge 1$ denotes the space dimension and
$\Rd$ the Euclidean space.
For $R>0$, we write
\[
B_R:=\{x\in\Rd:\ |x|<R\}.
\]
The Lebesgue $L^p$ norm on $\Omega\subset\Rd$ is denoted by $\|\cdot\|_{L^p(\Omega)}$.
For bounded pointwise-defined functions, $\|u\|_\infty:=\sup_{x\in\Rd}|u(x)|$
is the pointwise supremum norm. It agrees with the essential supremum for
continuous functions; arbitrary bounded functions are considered in
$\ell^\infty(\Rd)$, without identifying functions that agree almost everywhere.

\paragraph{Continuous operators.}
For a differentiable function $V:\Rd\to\mathbb{R}$, we define
\[
\nabla V(x):=(\partial_1 V(x),\dots,\partial_d V(x)),
\qquad
\Delta V(x):=\sum_{i=1}^d \partial_{ii}V(x),
\]
as the gradient and the Laplacian of $V$, respectively.

\paragraph{Discrete differences.}
Fix a mesh size $h\in(0,1)$.
For $\phi:\Rd\to\mathbb{R}$ and $i=1,\dots,d$, define
\[
D_h^i\phi(x):=\frac{\phi(x+h e_i)-\phi(x)}{h},
\qquad
D_{-h}^i\phi(x):=\frac{\phi(x)-\phi(x-h e_i)}{h}.
\]
The centered discrete gradient and Laplacian are
\begin{align}
\nabla_h\phi(x)
:=&
\left(
\frac{\phi(x+h e_1)-\phi(x-h e_1)}{2h},
\;
\dots,
\;
\frac{\phi(x+h e_d)-\phi(x-h e_d)}{2h}
\right),\label{eq:grad_h}\\
\Delta_h\phi(x)
:=&
\sum_{i=1}^d
\frac{\phi(x+h e_i)-2\phi(x)+\phi(x-h e_i)}{h^2}.\label{eq:laplace_h}
\end{align}

\paragraph{Semi-discrete operator.}
Let $\alpha:\Rd\to A$ be a bounded policy and define
\[
f_\alpha(x):=f(x,\alpha(x))\in\Rd,
\qquad
c_\alpha(x):=c(x,\alpha(x))\in\mathbb{R},
\]
with component notation
\[
f_\alpha(x)=(f_{\alpha,1}(x),\dots,f_{\alpha,d}(x)).
\]
% ============================================================
\subsection{Classical policy improvement and its regularity obstruction}\label{sec:illposed}
A classical stationary PI scheme evaluates a policy $\alpha_n:\Rd\to A$ by
solving
\begin{equation}\label{eq:cont_eval}
\lambda V_n(x)-c(x,\alpha_n(x))
-\nabla V_n(x)\cdot f(x,\alpha_n(x))=0
\qquad\text{in }\Rd,
\end{equation}
and then improves it through
\begin{equation}\label{eq:cont_improve}
\alpha_{n+1}(x)=\alpha\bigl(x,\nabla V_n(x)\bigr),
\end{equation}
where $\alpha(x,p)\in\argmin_{a\in A}\{c(x,a)+p\cdot f(x,a)\}$.

{
Even when a frozen policy is regular enough for \eqref{eq:cont_eval} to have
a unique bounded viscosity solution, that solution need not have a classical
gradient. If it is Lipschitz, its gradient is guaranteed only almost
everywhere; Lipschitz regularity is itself not automatic for arbitrary frozen
policies under Assumption~\ref{ass:A}. Thus \eqref{eq:cont_improve} does not
define a pointwise update on the full class of viscosity solutions considered
here. Additional regularity or a different formulation is needed.

We use centered differences to define the improvement step from point values
and add artificial viscosity to make the difference stencil monotone. This
construction permits comparison without asserting classical differentiability
of the semi-discrete solution.
\begin{remark}[Roles of discretization and discount]
Centered differences define policy improvement pointwise, and the
artificial-viscosity term ensures stencil monotonicity. The positive discount
supplies strict properness in the value variable and the contraction used
below. Indeed, the difference operators annihilate constants, so adding a
constant $k$ to the value changes the Bellman residual by $\lambda k$.
These roles underlie the comparison and convergence arguments.
\end{remark}
}

\section{Semi-discrete discounted HJB}\label{sec:scheme}

We now introduce a monotone {difference-quotient} approximation of the stationary
discounted Hamilton--Jacobi--Bellman equation
\begin{equation}\label{eq:HJB_cont}
\lambda V(x) + H(x,\nabla V(x)) = 0,
\qquad x\in\Rd.
\end{equation}

The goal is to retain the resolvent structure induced by the discount factor
$\lambda$, while regularizing gradients and restoring monotonicity at the
discrete level. As discussed in Section~\ref{sec:illposed}, this is essential
for obtaining a well-defined policy iteration map and stable convergence.

% ------------------------------------------------------------
\subsection{Definition of the semi-discrete scheme}

Fix a mesh size $h\in(0,1)$.
We replace the continuous gradient $\nabla$ by the centered discrete gradient
$\nabla_h$ defined in \eqref{eq:grad_h}, and introduce a discrete artificial
viscosity term of order $O(h)$. 
The semi-discrete stationary equation reads
\begin{equation}\label{eq:HJBh_new}
\lambda V^h(x)
+ H\big(x,\nabla_h V^h(x)\big)
= N h \Delta_h V^h(x),
\qquad x\in\Rd.
\end{equation}
The additional term $Nh\,\Delta_h V^h$ acts as a discrete artificial viscosity of order $O(h)$.
For every smooth test function $\phi$, we have $\nabla_h\phi\to\nabla\phi$
and $h\Delta_h\phi\to0$ locally uniformly as $h\to0$.
Thus \eqref{eq:HJBh_new} is consistent with the continuous HJB equation
\eqref{eq:HJB_cont}.

We introduce the semi-discrete affine residual operator
\begin{equation}\label{eq:Lalpha}
\mathcal{L}_\alpha^h U(x)
:=
\lambda U(x)
- c_\alpha(x)
- f_\alpha(x)\cdot \nabla_h U(x)
- N h\,\Delta_h U(x),
\end{equation}
for bounded policies $\alpha:\Rd\to A$ and then define the nonlinear Bellman operator
\begin{equation}\label{eq:Fh_def}
F_h[U](x):=\sup_{\alpha\in A}\mathcal L_\alpha^h U(x).
\end{equation}

With this notation, \eqref{eq:HJBh_new} is equivalently written as
\begin{equation}\label{eq:Bellman_form}
F_h[V^h](x)=0,
\qquad x\in\Rd.
\end{equation}

{%
\begin{remark}[terminology]\label{rem:terminology}
Equation \eqref{eq:HJBh_new} involves only shifted point values of $V^h$, so no
notion of viscosity solution is needed for it: by a solution we always mean a
bounded function satisfying \eqref{eq:HJBh_new} pointwise, and sub- and
supersolutions are understood in the same pointwise sense. Viscosity solutions
are used in this paper only for the continuous equation \eqref{eq:HJB_cont}.

{%
We also make precise what ``semi-discrete'' refers to. Only the differential
operator is discretized: $\nabla V$ is replaced by centered difference quotients
of step $h$. The state space is not discretized, and \eqref{eq:HJBh_new} is
imposed at every $x\in\Rd$, so the unknown $V^h$ is a bounded function on $\Rd$
rather than a vector of grid values, and \eqref{eq:HJBh_new} is a nonlocal
functional equation rather than a finite-dimensional system. A genuine grid, and
a truncated computational domain, appear only in the numerical experiments of
Section~\ref{sec:exp}.}
\end{remark}
}

The additional term $-Nh\Delta_h V^h$ ensures monotonicity of the
finite-difference stencil, which is crucial for comparison and stability.
Pointwise policy improvement uses the centered discrete gradient and does not
require a classical smoothing property of this nonlocal equation.
To guarantee monotonicity of the discrete operator, the artificial viscosity must
dominate the centered drift term. A sufficient condition is
\begin{equation}\label{eq:Nchoice_new}
N \;\ge\; \max\Big\{1,\frac{\|f\|_{L^\infty(\Rd\times A)}}{2}\Big\}.
\end{equation}
Under \eqref{eq:Nchoice_new}, the coefficients of the stencil values
$U(x\pm h e_i)$ in $\mathcal L_\alpha^h U(x)$, namely
{$-\big(\tfrac Nh\pm\tfrac{f_{\alpha,i}(x)}{2h}\big)$, are
nonpositive; equivalently, the neighbour weights
$\tfrac Nh\pm\tfrac{f_{\alpha,i}(x)}{2h}$ of the fixed-point map $T_\alpha$
defined in \eqref{eq:Talpha_def} below are nonnegative}.
Hence the scheme is monotone in the sense of finite-difference theory.
As a consequence, a discrete comparison principle holds, as established in
Proposition~\ref{prop:comparison}. 
% ------------------------------------------------------------

\subsection{Policy iteration for the semi-discrete equation}
Since \eqref{eq:HJBh_new} can be written in the Bellman form
\[
F_h[V](x) = \sup_{a\in A} \mathcal L_a^h V(x) = 0,
\]
the problem admits a dynamic programming structure.
Accordingly, we employ a Howard-type policy iteration scheme,
which alternates between policy evaluation (a linear resolvent problem)
and policy improvement (a pointwise maximization step).

\paragraph{Initialization.}
Choose an initial bounded {continuous} policy $\alpha_0:\Rd\to A$.
{%
Continuity is preserved along the iteration: if $\alpha_n$ is continuous, then
$V_n^h\in C_b(\Rd)$ by Corollary~\ref{cor:banach}, hence $\nabla_hV_n^h$ is
continuous and, by the Lipschitz continuity of $\alpha(x,p)$ in
Assumption~\ref{ass:policy}, so is
$\alpha_{n+1}=\alpha(\cdot,\nabla_hV_n^h(\cdot))$.
No global Lipschitz bound on $V_n^h$ is needed for the fixed-$h$ analysis;
continuity suffices.}

\paragraph{Policy evaluation.}
For a given policy $\alpha_n$, let $V_n^h:\Rd\to\mathbb R$ be the (bounded) solution of
\begin{equation}\label{eq:PI_eval_h}
\mathcal{L}_{\alpha_n}^{h} V_n^h(x)=0,
\qquad x\in\Rd.
\end{equation}
This corresponds to solving a linear resolvent equation associated with the frozen policy $\alpha_n$, which defines a contraction mapping due to the presence of the discount term $\lambda$.

\paragraph{Policy improvement.}
Define the next policy by
\begin{equation}\label{eq:PI_improve_h}
\alpha_{n+1}(x)=\alpha\bigl(x,\nabla_h V_n^h(x)\bigr),
\qquad x\in\Rd.
\end{equation}
Note that since $\nabla_h V_n^h(x)$ depends only on the point values
$V_n^h(x\pm h e_i)$, the update is well-defined pointwise without requiring differentiability of $V_n^h$. This step enforces the pointwise optimality condition in the Bellman operator, and corresponds to a greedy policy improvement step.

\paragraph{Fixed point.}
A function $V^h:\Rd\to\mathbb{R}$ satisfying \eqref{eq:HJBh_new}
is called the semi-discrete value function. {In view of the Bellman formulation \eqref{eq:Bellman_form}, $V^h$ is the unique \emph{zero} of the nonlinear operator $F_h$; it is the associated resolvent operator $T$ of Section~\ref{sec:resolvent} that admits $V^h$ as its unique fixed point.}

Starting from an initial policy $\alpha_0$, the policy iteration scheme
generates a sequence $\{V_n^h\}_{n\ge0}$ through alternating evaluation
and improvement steps.
Under the monotonicity condition \eqref{eq:Nchoice_new},
this sequence is well defined and satisfies
\[
V_{n+1}^h \le V_n^h \qquad \text{in }\Rd.
\]
{%
Moreover, the sequence is uniformly bounded in $\ell^\infty(\Rd)$ and converges
uniformly on $\Rd$ to the unique bounded solution $V^h$ of \eqref{eq:HJBh_new}.
Both the boundedness and the convergence are consequences of the resolvent
structure rather than of monotonicity: rewriting \eqref{eq:HJBh_new} as a
fixed-point equation for a contraction $T$ on $\ell^\infty(\Rd)$, as we do in
Section~\ref{sec:resolvent}, yields existence and uniqueness from the Banach
fixed-point theorem and geometric convergence from the contraction factor.
Monotonicity is used separately, to show that the iterates approach $V^h$ from
above.

This resolvent interpretation provides the basis for the subsequent convergence analysis, where the contraction arises from the discount factor $\lambda$ competing with the total stencil weight $2dN/h$.}

\section{Structural properties of the semi-discrete operator}
The following structural properties place the scheme within the framework of monotone approximation schemes for Hamilton--Jacobi equations \cite{BarlesSouganidis1991}.
\subsection{Monotonicity and comparison}\label{sec:tools}
\begin{lemma}\label{lem:mono_operator}
Assume \eqref{eq:Nchoice_new} and let $\alpha:\Rd\to A$ be fixed and $\mathcal L_\alpha^h$ defined in \eqref{eq:Lalpha}. 
% Define $\mathcal L_\alpha^h$ by
% \[
% \mathcal{L}_\alpha^h U(x)
% :=
% \lambda U(x)
% - c_\alpha(x)
% - f_\alpha(x)\cdot \nabla_h U(x)
% - N h\,\Delta_h U(x).
% \]
Then for each $x\in\Rd$, $\mathcal L_\alpha^h U(x)$ is
\emph{nondecreasing} in the central value $U(x)$ and \emph{nonincreasing} in each neighbor value
$U(x\pm h e_i)$. Equivalently, for functions $U,V:\Rd\to\mathbb R$:
\begin{enumerate}[label=(\roman*),leftmargin=18pt]
\item If $U(x)\le V(x)$ and $U(x\pm h e_i)=V(x\pm h e_i)$ for all $i$, then
\[
\mathcal L_\alpha^h U(x)\le \mathcal L_\alpha^h V(x).
\]
\item If $U(x)=V(x)$ and $U(x\pm h e_i)\le V(x\pm h e_i)$ for all $i$, then
\[
\mathcal L_\alpha^h U(x)\ge \mathcal L_\alpha^h V(x).
\]
\end{enumerate}
Moreover, the Bellman operator $F_h[U](x):=\sup_{a\in A}\mathcal L_a^h U(x)$ inherits the same
(monotone) dependence on stencil values.
\end{lemma}

\begin{proof}
% Fix $x\in\Rd$. Expand the discrete gradient and Laplacian:
% \[
% f_\alpha(x)\cdot \nabla_h U(x)
% =\sum_{i=1}^d \frac{f_{\alpha,i}(x)}{2h}\Big(U(x+h e_i)-U(x-h e_i)\Big),
% \]
% and
% \[
% -Nh\Delta_h U(x)
% =
% -\sum_{i=1}^d \frac{N}{h}\Big(U(x+h e_i)-2U(x)+U(x-h e_i)\Big).
% \]
% Therefore,
% \begin{align*}
% \mathcal L_\alpha^h U(x)
% &=
% \lambda U(x)-c_\alpha(x)
% -\sum_{i=1}^d \frac{f_{\alpha,i}(x)}{2h}U(x+h e_i)
% +\sum_{i=1}^d \frac{f_{\alpha,i}(x)}{2h}U(x-h e_i)\\
% &\quad
% -\sum_{i=1}^d \frac{N}{h}U(x+h e_i)
% +\sum_{i=1}^d \frac{2N}{h}U(x)
% -\sum_{i=1}^d \frac{N}{h}U(x-h e_i).
% \end{align*}
Expanding the discrete gradient and Laplacian and then collecting the coefficients yields the stencil form
\begin{equation}\label{eq:stencil_L}
\mathcal L_\alpha^h U(x)
=
\Big(\lambda+\frac{2dN}{h}\Big)U(x)
-c_\alpha(x)
+\sum_{i=1}^d a_i^+(x)\,U(x+h e_i)
+\sum_{i=1}^d a_i^-(x)\,U(x-h e_i),
\end{equation}
where
\[
a_i^+(x):=-\frac{N}{h}-\frac{f_{\alpha,i}(x)}{2h},
\qquad
a_i^-(x):=-\frac{N}{h}+\frac{f_{\alpha,i}(x)}{2h}.
\]  
By \eqref{eq:Nchoice_new}, we have $N\ge \|f\|_\infty/2$, hence
\[
a_i^+(x)\le -\frac{N}{h}+\frac{|f_{\alpha,i}(x)|}{2h}\le 0,
\qquad
a_i^-(x)\le -\frac{N}{h}+\frac{|f_{\alpha,i}(x)|}{2h}\le 0.
\]
On the other hand, the central coefficient satisfies
\[
\lambda+\frac{2dN}{h}>0.
\]
Thus, in \eqref{eq:stencil_L}, increasing \(U(x)\) increases \(\mathcal L_\alpha^h U(x)\), whereas increasing any neighboring value \(U(x \pm h e_i)\) decreases \(\mathcal L_\alpha^h U(x)\); hence (i)--(ii) follow.

Finally, since the pointwise supremum of functions that are nondecreasing (resp.\ nonincreasing)
in a given variable remains nondecreasing (resp.\ nonincreasing) in that variable, the Bellman
operator $F_h[U](x)=\sup_{a\in A}\mathcal L_a^h U(x)$ inherits the same monotonicity property.
\end{proof}
\begin{proposition}\label{prop:comparison}
Assume \eqref{eq:Nchoice_new}. 
{Let $U:\Rd\to\mathbb R$ be bounded and lower semicontinuous,
and let $\tilde U:\Rd\to\mathbb R$ be bounded and upper semicontinuous.}
Assume that $U$ is a supersolution and $\tilde U$ is a subsolution {(in the pointwise sense of Remark~\ref{rem:terminology})} of
\begin{equation}\label{eq:Bellman_eq}
F_h[W](x)=0 \qquad \text{in }\Rd,
\end{equation}
where
\[
F_h[W](x)
=
\sup_{a\in A}
\Big\{
\lambda W(x)
- c(x,a)
- f(x,a)\cdot \nabla_h W(x)
- Nh\Delta_h W(x)
\Big\}.
\]
Then $\tilde U\le U$ in $\Rd$.
\end{proposition}

\begin{proof}
We argue by contradiction. Suppose
\[
m:=\sup_{x\in\Rd}(\tilde U(x)-U(x))>0.
\]
Let $\varphi(x):=\sqrt{1+|x|^2}$ and for $\delta>0$ define
\[
\Phi_\delta(x)
:=
\tilde U(x)-U(x)-\delta\varphi(x).
\]
{Since $\tilde U$ is upper semicontinuous and $U$ is lower
semicontinuous, $\Phi_\delta$ is upper semicontinuous; as $\varphi(x)\to\infty$
when $|x|\to\infty$ and $\tilde U,U$ are bounded, $\Phi_\delta$ attains its
maximum at some $x_\delta\in\Rd$.}
Set $M_\delta:=\Phi_\delta(x_\delta)$. Note that $M_\delta\to m$ as $\delta\downarrow0$,
and in particular $M_\delta>0$ for all sufficiently small $\delta$.
Define
\[
U^\delta(x):=U(x)+M_\delta+\delta\varphi(x).
\]
By construction,
\[
U^\delta(x_\delta)=\tilde U(x_\delta),
\]
and since $x_\delta$ maximizes $\Phi_\delta$,
\begin{equation}\label{eq:Ud_dominate}
\tilde U(y)\le U^\delta(y)
\qquad\text{for all }y\in\Rd.
\end{equation}
In particular,
\[
\tilde U(x_\delta\pm he_i)\le U^\delta(x_\delta\pm he_i)
\quad\text{for all }i=1, \cdots, d.
\]
Therefore, applying Lemma~\ref{lem:mono_operator} yields
\begin{equation}\label{eq:mono_step}
F_h[\tilde U](x_\delta)\ge F_h[U^\delta](x_\delta).
\end{equation}
Since $\tilde U$ is a subsolution, $F_h[\tilde U](x_\delta)\le0$ and hence
\begin{equation}\label{eq:Ud_ineq}
0\ge F_h[U^\delta](x_\delta).
\end{equation}
Since $\nabla_h$ and $\Delta_h$ map constants to zero, we have
\begin{equation}\label{eq:shift}
F_h[U^\delta](x_\delta)
=
F_h[U+\delta\varphi](x_\delta)
+
\lambda M_\delta.
\end{equation}
For each $a\in A$, define
\[
\mathcal S_a^h[\psi](x)
:=
- f(x,a)\cdot\nabla_h\psi(x)
- Nh\,\Delta_h\psi(x).
\]
Then, we have 
\[
F_h[W](x)
=
\lambda W(x)+\sup_{a\in A}
\big(
- c(x,a)
+
\mathcal S_a^h[W](x)
\big).
\]
Using $\sup(A+B)\ge \sup A+\inf B$, we obtain
\begin{equation}\label{eq:perturb}
\begin{aligned}
F_h[U+\delta\varphi](x_\delta)&=\lambda U(x_\delta)+\delta\lambda \varphi(x_\delta)+\sup_{a\in A} \big(-c(x_\delta,a)+\mathcal S_a^h[U](x_\delta)+\delta \mathcal S_a^h[\varphi](x_\delta)\big)\\
&\ge
F_h[U](x_\delta)
+
\delta\lambda\varphi(x_\delta)
+
\delta\inf_{a\in A}\mathcal S_a^h[\varphi](x_\delta).
\end{aligned}
\end{equation}
Since $\varphi$ has bounded first and second derivatives,
$\nabla_h\varphi$ and $\Delta_h\varphi$ are bounded uniformly in $h$.
Moreover, since $f$ is bounded and $N$ is fixed,
the term $Nh\,\Delta_h\varphi$ is uniformly bounded for $h\in(0,1)$.
Therefore, there exists $C>0$, independent of $h$ and $\delta$, such that
\[
\inf_{a\in A}\mathcal S_a^h[\varphi](x)\ge -C
\qquad\text{for all }x\in\Rd.
\]
Hence from \eqref{eq:perturb} and the fact that $\varphi(x)\ge 0$ for all $x\in\Rd$, we have
\[
F_h[U+\delta\varphi](x_\delta)
\ge
F_h[U](x_\delta)-C\delta.
\]
Since $U$ is a supersolution,
$F_h[U](x_\delta)\ge0$.
Therefore,
\[
F_h[U+\delta\varphi](x_\delta)\ge -C\delta.
\]
Substituting into \eqref{eq:shift} and using \eqref{eq:Ud_ineq}, we conclude
\[
0
\ge
F_h[U^\delta](x_\delta)
=
F_h[U+\delta\varphi](x_\delta)
+
\lambda M_\delta
\ge
-C\delta+\lambda M_\delta.
\]
Thus
\[
\lambda M_\delta\le C\delta.
\]
Letting $\delta\downarrow0$ yields
\[
\limsup_{\delta\downarrow0}M_\delta\le0,
\]
contradicting $M_\delta\to m>0$.
Therefore $m\le0$, and hence $\tilde U\le U$ in $\Rd$.
\end{proof}

% ============================================================
{
\subsection{Resolvent form and contraction}\label{sec:resolvent}
Before turning to well-posedness, we rewrite the semi-discrete operators in
resolvent form. This form makes the contraction structure explicit and yields
existence and uniqueness for both the frozen-policy equation and the nonlinear
Bellman equation directly from the Banach fixed-point theorem, without appealing
to Perron's method or to any compactness argument.

For $a\in A$ and bounded $U:\Rd\to\mathbb R$, set
\begin{equation}\label{eq:Talpha_def}
(\mathcal T_a U)(x)
:=
\frac{
c(x,a)
+\displaystyle\sum_{i=1}^d\Bigl(\frac Nh+\frac{f_i(x,a)}{2h}\Bigr)U(x+he_i)
+\displaystyle\sum_{i=1}^d\Bigl(\frac Nh-\frac{f_i(x,a)}{2h}\Bigr)U(x-he_i)
}{
\lambda+\frac{2dN}{h}
},
\end{equation}
and, for a bounded policy $\alpha:\Rd\to A$,
\begin{equation}\label{eq:T_def}
(T_\alpha U)(x):=(\mathcal T_{\alpha(x)}U)(x),
\qquad
(TU)(x):=\inf_{a\in A}(\mathcal T_aU)(x).
\end{equation}

\begin{lemma}\label{lem:contraction}
Assume Assumption~\ref{ass:A} and \eqref{eq:Nchoice_new}, and set
\[
\beta_h:=\frac{2dN/h}{\lambda+2dN/h}\in(0,1).
\]
{Let
\[
\ell^\infty(\Rd):=\Bigl\{U:\Rd\to\mathbb R:\ \|U\|_\infty:=\sup_{x\in\Rd}|U(x)|<\infty\Bigr\},
\]
which is a Banach space for $\|\cdot\|_\infty$.}
Then, for all {$U,W\in\ell^\infty(\Rd)$}, all $a\in A$ and all $x\in\Rd$,
\begin{equation}\label{eq:L_T_relation}
\mathcal L_a^h U(x)
=
\Bigl(\lambda+\frac{2dN}{h}\Bigr)\bigl(U(x)-(\mathcal T_aU)(x)\bigr),
\qquad
F_h[U](x)
=
\Bigl(\lambda+\frac{2dN}{h}\Bigr)\bigl(U(x)-T(U)(x)\bigr).
\end{equation}
In particular
\[
\mathcal L^h_\alpha U=0\iff U=T_\alpha U,
\qquad\qquad
F_h[U]=0\iff U=T(U),
\]
so that the solutions of the frozen-policy equation and of the Bellman equation
are exactly the fixed points of $T_\alpha$ and of $T$.
Moreover, for every bounded policy $\alpha:\Rd\to A$, both $T_\alpha$ and $T$
map {$\ell^\infty(\Rd)$} into itself, are monotone, and satisfy
\begin{equation}\label{eq:contraction}
\|T_\alpha U-T_\alpha W\|_\infty\le\beta_h\|U-W\|_\infty,
\qquad
\|T U-T W\|_\infty\le\beta_h\|U-W\|_\infty.
\end{equation}
{If in addition $\alpha$ is continuous, then $T_\alpha$ maps $C_b(\Rd)$ into
itself; and $T$ maps $C_b(\Rd)$ into itself for every $U\in C_b(\Rd)$, since $A$
is compact and $(x,a)\mapsto(\mathcal T_aU)(x)$ is jointly continuous.}
\end{lemma}

\begin{proof}
Expanding $\nabla_h$ and $\Delta_h$ in \eqref{eq:Lalpha} and collecting the
coefficients of $U(x\pm he_i)$ gives
\[
\mathcal L_a^hU(x)
=
\Bigl(\lambda+\frac{2dN}{h}\Bigr)U(x)
-c(x,a)
-\sum_{i=1}^d\Bigl(\frac Nh+\frac{f_i(x,a)}{2h}\Bigr)U(x+he_i)
-\sum_{i=1}^d\Bigl(\frac Nh-\frac{f_i(x,a)}{2h}\Bigr)U(x-he_i),
\]
which is the first identity in \eqref{eq:L_T_relation}. Since the prefactor
$\lambda+\frac{2dN}{h}$ is positive and independent of $a$, and $U(x)$ does not
depend on $a$,
\[
F_h[U](x)=\sup_{a\in A}\mathcal L^h_aU(x)
=\Bigl(\lambda+\frac{2dN}{h}\Bigr)\Bigl(U(x)-\inf_{a\in A}(\mathcal T_aU)(x)\Bigr),
\]
which is the second identity.

By \eqref{eq:Nchoice_new} the weights
$w_i^\pm(x,a):=\frac Nh\pm\frac{f_i(x,a)}{2h}$ are nonnegative, and
\[
\sum_{i=1}^d\bigl(w_i^+(x,a)+w_i^-(x,a)\bigr)=\frac{2dN}{h}.
\]
Hence $\mathcal T_a$ is an affine map whose coefficients are nonnegative with
total mass $\frac{2dN}{h}\big/\bigl(\lambda+\frac{2dN}{h}\bigr)$, so it is
monotone and
\[
|\mathcal T_aU(x)-\mathcal T_aW(x)|
\le
\frac{2dN/h}{\lambda+2dN/h}\,\|U-W\|_\infty
=\beta_h\|U-W\|_\infty .
\]
Taking $a=\alpha(x)$ and then the supremum over $x$ gives the first estimate in
\eqref{eq:contraction}. For $T$ we use
$|\inf_a p_a-\inf_a q_a|\le\sup_a|p_a-q_a|$ pointwise, which yields the second.
Boundedness of $T_\alpha U$ and $TU$ follows from
$\|\mathcal T_aU\|_{\infty}\le\bigl(\|c\|_{L^\infty}+\tfrac{2dN}{h}\|U\|_{\infty}\bigr)\big/\bigl(\lambda+\tfrac{2dN}{h}\bigr)$.
{%
Finally, if $U\in C_b(\Rd)$ and $\alpha$ is continuous, then $T_\alpha U$ is a
finite sum of continuous functions and hence lies in $C_b(\Rd)$; and since $A$ is
compact and $(x,a)\mapsto(\mathcal T_aU)(x)$ is continuous by \ref{A1}, the
infimum $TU=\inf_{a\in A}\mathcal T_aU$ is attained and continuous in $x$.}
\end{proof}

\begin{corollary}\label{cor:banach}
Under the assumptions of Lemma~\ref{lem:contraction}, for every bounded policy
$\alpha$ the frozen-policy equation $\mathcal L^h_\alpha V=0$ has a unique
solution in {$\ell^\infty(\Rd)$}, and the Bellman equation $F_h[V^h]=0$ has a
unique solution {$V^h\in\ell^\infty(\Rd)$}; equivalently, $V^h$ is the unique fixed
point of $T$.
Both follow from the Banach fixed-point theorem applied to \eqref{eq:contraction}
on {$\bigl(\ell^\infty(\Rd),\|\cdot\|_\infty\bigr)$}, and the corresponding Picard
iterates converge uniformly on $\Rd$ at the geometric rate $\beta_h$.
{Uniqueness therefore holds among all bounded functions, not only among the
continuous ones.

If moreover $\alpha$ is continuous, then $C_b(\Rd)$ is a closed subspace left
invariant by $T_\alpha$ and by $T$, so both fixed points lie in $C_b(\Rd)$.}
\end{corollary}

\begin{corollary}\label{cor:frozen_comparison}
Under the assumptions of Lemma~\ref{lem:contraction}, let $\alpha$ be any
bounded policy and let $U,W\in\ell^\infty(\Rd)$ satisfy
$\mathcal L_\alpha^hU\le0\le\mathcal L_\alpha^hW$ pointwise.
Then $U\le W$.
\end{corollary}
\begin{proof}
By \eqref{eq:L_T_relation}, $U\le T_\alpha U$ and $W\ge T_\alpha W$.
Monotonicity gives $U\le T_\alpha^kU$ and $W\ge T_\alpha^kW$ for every $k$.
Both Picard sequences converge uniformly to the unique fixed point of
$T_\alpha$, proving the assertion without any regularity requirement on
$x\mapsto\alpha(x)$.
\end{proof}

Throughout, we therefore speak of $V^h$ as the unique \emph{zero} of $F_h$ and,
equivalently, as the unique \emph{fixed point} of $T$.
}
% ============================================================

% ============================================================
\subsection{Well-posedness and monotonicity of semi-discrete PI}\label{sec:wellposed}
We now establish the well-posedness of the semi-discrete policy iteration scheme
and its basic structural properties. In particular, we show that each policy evaluation step admits a unique bounded solution, and that the resulting value sequence generated by policy iteration is monotone and converges to the unique solution of the semi-discrete Bellman equation.

\begin{proposition}\label{prop:wellposed}
Suppose that Assumption~\ref{ass:A} and \eqref{eq:Nchoice_new} hold.
For each bounded {continuous} policy $\alpha_n$, there exists a unique bounded {solution} $V_n^h\in C_b(\Rd)$
to \eqref{eq:PI_eval_h}. Moreover, the following estimate
\begin{equation}\label{eq:Linfty_bound}
\norm{V_n^h}_{L^\infty(\Rd)}\le \frac{\norm{c}_{L^\infty(\Rd\times A)}}{\lambda}
\end{equation}
holds for all $n\in \mathbb{N}$ and $h\in (0,1)$.
\end{proposition}

\begin{proof}
% Let $\mathcal{L}_{\alpha_n}^h U:=\lambda U + c(\cdot,\alpha_n(\cdot)) + \nabla_h U\cdot f(\cdot,\alpha_n(\cdot)) + Nh\Delta_h U$.
% Monotonicity (from \eqref{eq:Nchoice_new}) implies a comparison principle for bounded sub/super-solutions.
Let $M:=\|c\|_{L^\infty(\Rd\times A)}/\lambda$.
Then $\mathcal{L}_{\alpha_n}^h(M)\ge 0$ and $\mathcal{L}_{\alpha_n}^h(-M)\le 0$.
Hence $M$ is a supersolution and $-M$ is a subsolution.
By Corollary~\ref{cor:frozen_comparison}, the estimate \eqref{eq:Linfty_bound} follows.
{Existence and uniqueness follow from Corollary~\ref{cor:banach}:
by Lemma~\ref{lem:contraction}, solving \eqref{eq:PI_eval_h} amounts to finding
the fixed point of the contraction $T_{\alpha_n}$ on $\ell^\infty(\Rd)$, so the
Banach fixed-point theorem applies; since $\alpha_n$ is continuous, the solution
lies in $C_b(\Rd)$.}
\end{proof}

To analyze the policy improvement step and the convergence of the associated policy sequence, we introduce an additional structural assumption on the policy map.

\begin{assumption}\label{ass:policy}
For each $(x,p)\in\Rd\times\Rd$, the minimization problem
\[
\min_{a\in A}\{c(x,a)+p\cdot f(x,a)\}
\]
admits a unique minimizer.
Moreover the induced policy map
\[
\alpha(x,p):=\argmin_{a\in A}\{c(x,a)+p\cdot f(x,a)\}
\]
is globally Lipschitz continuous in $(x,p)$. We denote its global Lipschitz constant by $L_\alpha > 0$.
\end{assumption}

{%
Under this additional assumption we establish monotonicity of the policy
iteration sequence. Convergence is \emph{not} deduced from monotonicity: it
follows separately from the contraction estimate \eqref{eq:contraction}, in
Theorem~\ref{thm:fixed_h_supnorm}.

\begin{proposition}\label{prop:mono}
Suppose that Assumption~\ref{ass:A}, \ref{ass:policy} hold and $N$ satisfies \eqref{eq:Nchoice_new}, and let $V^h$ be the unique bounded solution of \eqref{eq:HJBh_new} given by Corollary~\ref{cor:banach}. Then for all $n\ge0$,
\[
V^h(x)\le V_{n+1}^h(x)\le V_n^h(x)\qquad \forall x\in\Rd.
\]
\end{proposition}
}

\begin{proof}
By optimality of the improvement step \eqref{eq:PI_improve_h},
\[
- c(x,\alpha_{n+1}(x))
- \nabla_h V_n^h(x)\cdot f(x,\alpha_{n+1}(x))
\ge
- c(x,\alpha_n(x))
- \nabla_h V_n^h(x)\cdot f(x,\alpha_n(x)).
\]
Using the evaluation equation \eqref{eq:PI_eval_h} satisfied by $V_n^h$,
we obtain
\[
\mathcal{L}_{\alpha_{n+1}}^h V_n^h(x)
\ge
\mathcal{L}_{\alpha_n}^h V_n^h(x)
=
0.
\]
Thus $V_n^h$ is a supersolution of the evaluation equation with policy $\alpha_{n+1}$.
Since $V_{n+1}^h$ is the unique solution of
\[
\mathcal{L}_{\alpha_{n+1}}^h V_{n+1}^h = 0,
\]
the frozen-policy comparison principle, Corollary~\ref{cor:frozen_comparison}, yields
$V_{n+1}^h \le V_n^h$  in $\Rd$.

{%
It remains to prove the lower bound $V^h\le V_n^h$. Since
$V^h=T(V^h)=\inf_{a\in A}\mathcal T_aV^h\le T_{\alpha_n}V^h$, the function $V^h$
is a subsolution of the fixed-point problem for $T_{\alpha_n}$. As
$T_{\alpha_n}$ is monotone by Lemma~\ref{lem:contraction}, the Picard iterates
$T_{\alpha_n}^kV^h$ are nondecreasing in $k$, and by \eqref{eq:contraction} they
converge uniformly to the unique fixed point $V_n^h$ of $T_{\alpha_n}$. Hence
\[
V^h\le T_{\alpha_n}^kV^h\ \xrightarrow[k\to\infty]{}\ V_n^h ,
\]
which gives $V^h\le V_n^h$ for every $n\ge0$.
}
\end{proof}

This establishes that the semi-discrete policy iteration scheme is well posed
and monotone. Convergence to the unique Bellman solution is proved in
Theorem~\ref{thm:fixed_h_supnorm}.

% ============================================================
\section{Convergence results}
In this section, we analyze the convergence properties of the semi-discrete
policy iteration scheme. We first establish geometric convergence of the value iterates for fixed mesh size $h$, based on a contraction property induced by the discounted resolvent structure. We then quantify the discretization error as $h\to0$, and combine the two results to obtain a unified error estimate that reveals the interaction between the iteration count and the spatial discretization.

\subsection{Fixed point arguments}\label{sec:expconv}
We combine the fixed-point formulation of policy evaluation with the
policy-improvement identity below. Together with monotonicity, the resolvent
contraction yields geometric convergence of the value iterates for fixed $h$.

{
\begin{lemma}
\label{lem:fixed_point_representation}
Assume Assumption~\ref{ass:A}, \ref{ass:policy} and \eqref{eq:Nchoice_new}, and
let $\mathcal T_a$, $T_\alpha$, $T$ be as in
\eqref{eq:Talpha_def}--\eqref{eq:T_def}.
If \(\alpha_{n+1}\) is defined by \eqref{eq:PI_improve_h}, then
\begin{equation}\label{eq:T_improvement_identity}
T(V_n^h)=T_{\alpha_{n+1}}V_n^h.
\end{equation}
\end{lemma}
}

The minimization defining $T$ is consistent with the maximization defining
$F_h$: by \eqref{eq:L_T_relation}, $F_h[U]$ is the positive multiple
$(\lambda+2dN/h)(U-TU)$.

\begin{proof}
For \(U=V_n^h\),
\[
(\mathcal T_aV_n^h)(x)
=
\frac{
c(x,a)
+\frac Nh\sum_{i=1}^d\bigl(V_n^h(x+he_i)+V_n^h(x-he_i)\bigr)
+f(x,a)\cdot \nabla_h V_n^h(x)
}{
\lambda+\frac{2dN}{h}
}.
\]
For fixed \(x\), the denominator and the middle term
\[
\frac Nh\sum_{i=1}^d\bigl(V_n^h(x+he_i)+V_n^h(x-he_i)\bigr)
\]
do not depend on \(a\). Hence minimizing \((\mathcal T_aV_n^h)(x)\) over \(a\in A\)
is equivalent to minimizing
\[
a\mapsto c(x,a)+f(x,a)\cdot \nabla_h V_n^h(x).
\]
{Assumption~\ref{ass:A}\ref{A2} guarantees that the minimum is
attained, and Assumption~\ref{ass:policy} that the minimizer is unique; by the
definition \eqref{eq:PI_improve_h} it is precisely \(\alpha_{n+1}(x)\).} Therefore,
\[
T(V_n^h)(x)
=
\inf_{a\in A}(\mathcal T_aV_n^h)(x)
=
(\mathcal T_{\alpha_{n+1}(x)}V_n^h)(x)
=
(T_{\alpha_{n+1}}V_n^h)(x).
\]
This proves \eqref{eq:T_improvement_identity}.
\end{proof}

\begin{theorem}
\label{thm:fixed_h_supnorm}
Assume Assumption~\ref{ass:A}{, \ref{ass:policy}} and \eqref{eq:Nchoice_new}. Fix \(h\in(0,1)\).
Let \(\{V_n^h\}_{n\ge0}\) be generated by \eqref{eq:PI_eval_h}--\eqref{eq:PI_improve_h},
and let \(V^h\) be the unique bounded solution of \eqref{eq:HJBh_new}. Then for 
$$
\beta_h:=\frac{2dN/h}{\lambda+2dN/h}\in(0,1)
$$
we have
\[
\|V_n^h-V^h\|_{L^\infty(\Rd)}
\le
\beta_h^{\,n}\,\|V_0^h-V^h\|_{L^\infty(\Rd)}\le \frac{2\|c\|_{L^\infty(\Rd\times A)}}{\lambda}\beta_h^n.
\]
\end{theorem}

\begin{proof}
{%
By Lemma~\ref{lem:contraction}, $T_\alpha$ and $T$ are monotone and are
contractions on {$\ell^\infty(\Rd)$} with factor $\beta_h$, and $V^h$ is the unique
fixed point of $T$.
}

By Lemma~\ref{lem:fixed_point_representation},
\[
T(V_n^h)=T_{\alpha_{n+1}}V_n^h,
\qquad
V_{n+1}^h=T_{\alpha_{n+1}}V_{n+1}^h,
\qquad
V^h=T(V^h).
\]
By Proposition~\ref{prop:mono}, we have
\[
V^h\le V_{n+1}^h\le V_n^h
\qquad \text{in }\Rd.
\]
Since \(T_{\alpha_{n+1}}\) is monotone,
\[
V_{n+1}^h
=
T_{\alpha_{n+1}}V_{n+1}^h
\le
T_{\alpha_{n+1}}V_n^h
=
T(V_n^h).
\]
Therefore,
\[
0\le V_{n+1}^h-V^h\le T(V_n^h)-T(V^h),
\]
Since each $\mathcal T_a$ is monotone and $T=\inf_{a\in A}\mathcal T_a$,
the operator $T$ is monotone. Taking the \(L^\infty\) norm and using the contraction of \(T\),
\[
\|V_{n+1}^h-V^h\|_{L^\infty(\Rd)}
\le
\beta_h\,\|V_n^h-V^h\|_{L^\infty(\Rd)}.
\]
Iteration gives
\[
\|V_n^h-V^h\|_{L^\infty(\Rd)}
\le
\beta_h^{\,n}\,\|V_0^h-V^h\|_{L^\infty(\Rd)}.
\]

{Finally, Proposition~\ref{prop:wellposed} bounds $\|V_0^h\|_{L^\infty}$,
while the same bound for $V^h$ follows by comparing \eqref{eq:HJBh_new} with the
constant barriers $\pm\|c\|_{L^\infty(\Rd\times A)}/\lambda$, which are
respectively a super- and a subsolution; together they give}
\[
\|V_0^h\|_{L^\infty(\Rd)}+\|V^h\|_{L^\infty(\Rd)}
\le
\frac{2\|c\|_{L^\infty(\Rd\times A)}}{\lambda},
\]
hence
\[
\|V_0^h-V^h\|_{L^\infty(\Rd)}
\le
\frac{2\|c\|_{L^\infty(\Rd\times A)}}{\lambda}.
\]
\end{proof}

{
\begin{proposition}\label{prop:local_quadratic}
Under the assumptions of Theorem~\ref{thm:fixed_h_supnorm}, let $L_{f,a}$
be a uniform Lipschitz constant of $f(x,\cdot)$, as supplied by \ref{A1}.
Then
\begin{equation}\label{eq:local_quadratic}
\|V_{n+1}^h-V^h\|_\infty
\le \frac{dL_{f,a}L_\alpha}{\lambda h^2}
       \|V_n^h-V^h\|_\infty^2.
\end{equation}
Consequently, the globally convergent sequence enters a local quadratic convergence
regime for every fixed $h$. The constant is not uniform as $h\downarrow0$.
\end{proposition}
\begin{proof}
At each $x$, set $p_n=\nabla_hV_n^h(x)$, $p_*=\nabla_hV^h(x)$,
$a'=\alpha(x,p_n)$, $a_*=\alpha(x,p_*)$, and
$Q_x(p,a)=c(x,a)+p\cdot f(x,a)$. Optimality of $a_*$ at $p_*$ and of $a'$ at
$p_n$ gives
\[
\begin{aligned}
0\le G_n(x)&:=Q_x(p_*,a')-Q_x(p_*,a_*)\\
&\le(p_*-p_n)\cdot\bigl(f(x,a')-f(x,a_*)\bigr)\\
&\le L_{f,a}L_\alpha|p_n-p_*|^2.
\end{aligned}
\]
Subtracting the Bellman equation for $V^h$ from the evaluation equation for
$V_{n+1}^h$ yields, with $e=V_{n+1}^h-V^h$,
\[
\lambda e-f(x,a')\cdot\nabla_he-Nh\Delta_he=G_n.
\]
Comparison with the constant barriers $\pm\|G_n\|_\infty/\lambda$ gives
$\|e\|_\infty\le\|G_n\|_\infty/\lambda$. This uses the same monotone
resolvent argument as Corollary~\ref{cor:frozen_comparison}, with source $G_n$.
Finally, $|\nabla_hu(x)|\le\sqrt d\,\|u\|_\infty/h$ gives
\eqref{eq:local_quadratic}. The geometric convergence already proved ensures
entry into the local regime.
\end{proof}
}

For the semi-discrete stationary problem, the convergence of policy iteration
is driven by the discounted resolvent structure.
More precisely, once the policy-evaluation equation is rewritten as a fixed-point problem
for the map $T_\alpha$, the contraction factor is
\[
\beta_h=\frac{2dN/h}{\lambda+2dN/h}\in(0,1).
\]
Thus the damping arises from the competition between the zeroth-order discount term
$\lambda$ and the total stencil weight $2dN/h$. In particular, the contraction weakens as $\lambda\downarrow0$, which is consistent with the greater difficulty of the undiscounted stationary problem.

\begin{remark}[Stationary discounted vs.\ finite-horizon PI]
The semi-discrete regularization used here is analogous in spirit to that of
\cite{TangTranZhang2025}: in both cases, monotone artificial viscosity is introduced
to restore comparison and to make the policy-improvement step well defined through
discrete gradients. The estimates used to control the iteration differ.
For finite-horizon parabolic HJB equations, the analysis is based on time evolution
and Gr\"onwall-type propagation.
For the stationary discounted equation, there is no time variable,
and the fixed-$h$ convergence is instead a consequence of the resolvent contraction
induced by the discount term $\lambda$.
Accordingly, the relevant constants deteriorate as $\lambda\downarrow0$,
rather than with the time horizon.
\end{remark}

The geometric convergence of the value iterates yields convergence of the associated policies. Indeed, since the policy map depends on the discrete gradient of the value function, the Lipschitz continuity of $\alpha(x,p)$ allows us to transfer the value convergence to the policy sequence.

\begin{corollary}\label{cor:policy_conv}
Under the same assumptions in Theorem~\ref{thm:fixed_h_supnorm}, let $\{V_n^h\}_{n\ge0}$ be generated by the semi-discrete PI
\eqref{eq:PI_eval_h}--\eqref{eq:PI_improve_h}, and let $V^h$ be the unique bounded
solution to \eqref{eq:HJBh_new}. Define
\[
\alpha_{n+1}(x):=\alpha\bigl(x,\nabla_h V_n^h(x)\bigr),
\qquad
\alpha^h(x):=\alpha\bigl(x,\nabla_h V^h(x)\bigr).
\]
Then
\begin{equation}\label{eq:policy_Linf_grad}
\|\alpha_{n+1}-\alpha^h\|_{L^\infty(\Rd)}
\le
L_\alpha\,\|\nabla_h(V_n^h-V^h)\|_{L^\infty(\Rd)},
\end{equation}
and
\begin{equation}\label{eq:policy_shift_bound}
\|\alpha_{n+1}-\alpha^h\|_{L^\infty(\Rd)}
\le
\frac{\sqrt d\,L_\alpha}{h}\|V_n^h-V^h\|_{L^\infty(\Rd)}.
\end{equation}
\end{corollary}

\begin{proof}
Recall from Assumption~\ref{ass:policy} that the policy map $\alpha$ is Lipschitz continuous with constant $L_\alpha$. Thus,
\[
|\alpha_{n+1}(x)-\alpha^h(x)|
=
\bigl|\alpha(x,\nabla_h V_n^h(x))-\alpha(x,\nabla_h V^h(x))\bigr|
\le
L_\alpha\,|\nabla_h(V_n^h-V^h)(x)|.
\]
Taking the supremum over $x\in\Rd$ yields \eqref{eq:policy_Linf_grad}.

Next, for any bounded $u$ and any $x\in\Rd$,
\[
|\nabla_h u(x)|
=
\frac{1}{2h}\left(\sum_{i=1}^d|u(x+he_i)-u(x-he_i)|^2\right)^{1/2}
\le
\frac{\sqrt d}{h}\|u\|_\infty.
\]
Applying this with $u=V_n^h-V^h$ and combining with \eqref{eq:policy_Linf_grad}
gives \eqref{eq:policy_shift_bound}.
\end{proof}

% ============================================================
\subsection{Discretization error: convergence of \texorpdfstring{$V^h\to V$}{Vh to V} as \texorpdfstring{$h\to 0$}{h to 0}}\label{sec:disc_error}

Let $V$ be the unique bounded viscosity solution of \eqref{eq:HJB}, and $V^h$ solve \eqref{eq:HJBh_new}. Under a discount condition ensuring that $V$ is Lipschitz, we obtain the canonical $\sqrt h$ approximation rate. We first recall a classical regularity result~\cite{tran2021hamilton,CrandallIshiiLions1992,BardiCapuzzo1997}.

\begin{lemma}
\label{lem:V_lipschitz}
Assume Assumption~\ref{ass:A} and suppose that
\[
\lambda > \Lip_x(f).
\]
Then the viscosity solution $V$ of
\[
\lambda V + H(x,\nabla V)=0
\]
is globally Lipschitz continuous, with
\[
\Lip(V)\le\frac{\Lip_x(c)}{\lambda-\Lip_x(f)}.
\]
\end{lemma}
{
\begin{proof}
Set $L_f=\Lip_x(f)$ and $L_c=\Lip_x(c)$. For any measurable control
$a(\cdot)$, the trajectories starting at $x,y$ satisfy
$|X_x^a(t)-X_y^a(t)|\le e^{L_ft}|x-y|$ by Gr\"onwall's inequality. Hence
\[
|J(x;a)-J(y;a)|
\le L_c|x-y|\int_0^\infty e^{-(\lambda-L_f)t}\,dt
=\frac{L_c}{\lambda-L_f}|x-y|.
\]
Taking infima over the same class of controls proves the estimate for $V$.
\end{proof}

We also record a uniform Lipschitz estimate for the semi-discrete solutions
under a stronger discount condition. It is a separate regularity result;
Theorem~\ref{thm:h_to_0} below uses only the preceding estimate for $V$.
}

{
\begin{lemma}
\label{lem:Vh_lipschitz}
Assume Assumption~\ref{ass:A}, \eqref{eq:Nchoice_new}, and
\[
\lambda>\sqrt d\,\Lip_x(f).
\]
For each $h\in(0,1)$, let $V^h\in C_b(\Rd)$ denote the unique bounded
{solution} of the semi-discrete equation
\begin{equation}\label{eq:HJBh_lip}
\lambda V^h(x)+H\bigl(x,\nabla_h V^h(x)\bigr)
=
Nh\,\Delta_h V^h(x),
\qquad x\in\Rd.
\end{equation}
Then there exists a constant $L>0$, independent of $h\in(0,1)$, such that
\[
|V^h(x)-V^h(y)|\le L|x-y|
\qquad\text{for all }x,y\in\Rd.
\]
More precisely,
\[
\sup_{h\in(0,1)}\Lip(V^h)
\le
\frac{\Lip_x(c)}
{\lambda-\sqrt d\,\Lip_x(f)}.
\]
\end{lemma}

\begin{proof}
Set
\[
L_c:=\Lip_x(c),
\qquad
L_f:=\Lip_x(f).
\]
Fix $h\in(0,1)$ and $z\in\Rd$, and define
\[
V_z^h(x):=V^h(x+z).
\]
By translation invariance of $\nabla_h$ and $\Delta_h$, the function
$V_z^h$ satisfies
\[
\lambda V_z^h(x)
+
H\bigl(x,\nabla_h V_z^h(x)\bigr)
-
Nh\,\Delta_h V_z^h(x)
=
R_z^h(x),
\]
where
\[
R_z^h(x)
:=
H\bigl(x,\nabla_h V_z^h(x)\bigr)
-
H\bigl(x+z,\nabla_h V_z^h(x)\bigr).
\]
For every $p\in\Rd$, the definition of $H$ gives
\begin{equation}\label{eq:H_x_growth}
|H(x,p)-H(y,p)|
\le
\bigl(L_c+L_f|p|\bigr)|x-y|.
\end{equation}

We first note that $V^h$ is globally Lipschitz for each fixed $h$,
with a constant that may initially depend on $h$.
{Indeed, the constants $\pm M$ with
$M:=\|c\|_{L^\infty(\Rd\times A)}/\lambda$ satisfy $F_h[M]\ge0$ and
$F_h[-M]\le0$, so they are respectively a super- and a subsolution of
\eqref{eq:HJBh_lip}, and Proposition~\ref{prop:comparison} gives}
\[
\|V^h\|_{L^\infty(\Rd)}
\le
\frac{\|c\|_{L^\infty(\Rd\times A)}}{\lambda},
\]
and hence
\[
\|\nabla_h V^h\|_{L^\infty(\Rd)}
\le
\frac{\sqrt d}{h}\|V^h\|_{L^\infty(\Rd)}.
\]
It follows from \eqref{eq:H_x_growth} that
\[
\|R_z^h\|_{L^\infty(\Rd)}
\le C_h|z|
\]
for some finite constant $C_h$.
The comparison principle applied to $V_z^h$ and
$V^h+C_h|z|/\lambda$ therefore yields
\[
|V^h(x+z)-V^h(x)|
\le
\frac{C_h}{\lambda}|z|.
\]
Thus
\[
K_h:=\Lip(V^h)<\infty.
\]

Since $V^h$ is $K_h$-Lipschitz, each component of its centered discrete
gradient satisfies
\[
\left|
\frac{V^h(x+he_i)-V^h(x-he_i)}{2h}
\right|
\le K_h,
\]
and consequently
\[
|\nabla_h V^h(x)|\le \sqrt d\,K_h.
\]
Using \eqref{eq:H_x_growth} once more, we obtain the improved estimate
\[
\|R_z^h\|_{L^\infty(\Rd)}
\le
\bigl(L_c+\sqrt d\,L_fK_h\bigr)|z|.
\]
Therefore, comparison gives
\[
|V^h(x+z)-V^h(x)|
\le
\frac{L_c+\sqrt d\,L_fK_h}{\lambda}|z|.
\]
Taking the supremum over $x$ and $z\ne0$, we conclude that
\[
K_h
\le
\frac{L_c+\sqrt d\,L_fK_h}{\lambda}.
\]
Since $\lambda>\sqrt d\,L_f$, it follows that
\[
K_h
\le
\frac{L_c}{\lambda-\sqrt d\,L_f},
\]
uniformly in $h\in(0,1)$.
\end{proof}
}

\begin{theorem}\label{thm:h_to_0}
Assume Assumption~\ref{ass:A}, \eqref{eq:Nchoice_new}, and
\[
\lambda>\Lip_x(f),
\]
with $N$ fixed independently of $h$.
Let $V$ be the unique bounded viscosity solution of
\begin{equation}\label{eq:HJB_cont_thm}
\lambda V(x)+H(x,\nabla V(x))=0
\qquad \text{in }\Rd,
\end{equation}
and, for each $h\in(0,1)$, let $V^h$ be the unique bounded {solution} of
\begin{equation}\label{eq:HJBh_cont_thm}
\lambda V^h(x)+H\bigl(x,\nabla_h V^h(x)\bigr)=Nh\,\Delta_h V^h(x)
\qquad \text{in }\Rd.
\end{equation}
Then there exists a constant $C>0$, independent of $h\in(0,1)$, such that
\[
\|V^h-V\|_{L^\infty(\Rd)}\le C\sqrt{h}.
\]
\end{theorem}

{
\begin{proof}
Write $U=V^h$, $L_c=\Lip_x(c)$, $L_f=\Lip_x(f)$, and $F=\|f\|_\infty$.
Lemma~\ref{lem:V_lipschitz} gives
$L:=L_c/(\lambda-L_f)$ as a Lipschitz bound for $V$.
The function $U$ is bounded and continuous by Corollary~\ref{cor:banach}.
Set $\nu=Nh$ and $\varphi(z)=\sqrt{1+|z|^2}$. The bounds
\[
|\nabla\varphi|\le1,\qquad |\nabla_h\varphi|\le\sqrt d,
\qquad |\Delta_h\varphi|\le d
\]
hold uniformly for $h\in(0,1)$. The Hamiltonian satisfies
\begin{equation}\label{eq:H_mixed_lipschitz}
|H(x,p)-H(y,q)|\le F|p-q|+(L_c+L_f|q|)|x-y|.
\end{equation}

Fix $\varepsilon>0$ and $\eta\in(0,1)$. To estimate $V-U$, let
$(\bar x,\bar y)$ maximize
\[
\Phi(x,y)=V(x)-U(y)-\frac{|x-y|^2}{2\varepsilon}
-\eta\bigl(\varphi(x)+\varphi(y)\bigr).
\]
The maximum exists by continuity, boundedness of $V,U$, and coercivity of
the localization term. Comparing with $(\bar y,\bar y)$ gives
\[
\frac{|\bar x-\bar y|^2}{2\varepsilon}
\le V(\bar x)-V(\bar y)
+\eta\bigl(\varphi(\bar y)-\varphi(\bar x)\bigr)
\le(L+\eta)|\bar x-\bar y|.
\]
Consequently, $p=(\bar x-\bar y)/\varepsilon$ satisfies $|p|\le2(L+\eta)$.
After adjusting additive constants to achieve contact, the function
$x\mapsto |x-\bar y|^2/(2\varepsilon)+\eta\varphi(x)$ touches $V$
globally from above, and
$y\mapsto-|\bar x-y|^2/(2\varepsilon)-\eta\varphi(y)$ touches $U$
globally from below. The viscosity subsolution inequality for $V$ and
stencil monotonicity for $U$ therefore give
\begin{align*}
\lambda V(\bar x)+H(\bar x,p+\eta\nabla\varphi(\bar x))&\le0,\\
\lambda U(\bar y)+H(\bar y,p-\eta\nabla_h\varphi(\bar y))
+\frac{d\nu}{\varepsilon}+\eta\nu\Delta_h\varphi(\bar y)&\ge0.
\end{align*}
Here the centered differences of the quadratic test functions are exact.
Subtracting yields
\begin{align}\label{eq:disc_error_forward}
\lambda[V(\bar x)-U(\bar y)]
&\le H(\bar y,p-\eta\nabla_h\varphi(\bar y))
-H(\bar x,p+\eta\nabla\varphi(\bar x))\nonumber\\
&\quad+\frac{d\nu}{\varepsilon}+\eta\nu\Delta_h\varphi(\bar y).
\end{align}

For the reverse estimate, let $(\hat x,\hat y)$ maximize
\[
\widetilde\Phi(x,y)=U(x)-V(y)-\frac{|x-y|^2}{2\varepsilon}
-\eta\bigl(\varphi(x)+\varphi(y)\bigr).
\]
Comparing with $(\hat x,\hat x)$ gives
\[
\frac{|\hat x-\hat y|^2}{2\varepsilon}
\le V(\hat x)-V(\hat y)
+\eta\bigl(\varphi(\hat x)-\varphi(\hat y)\bigr)
\le(L+\eta)|\hat x-\hat y|.
\]
Thus $p=(\hat x-\hat y)/\varepsilon$ again obeys $|p|\le2(L+\eta)$,
using only the Lipschitz bound for $V$. The quadratic tests now touch $U$
globally from above and $V$ globally from below. Stencil monotonicity and the
viscosity supersolution inequality give
\begin{align*}
\lambda U(\hat x)+H(\hat x,p+\eta\nabla_h\varphi(\hat x))
-\frac{d\nu}{\varepsilon}-\eta\nu\Delta_h\varphi(\hat x)&\le0,\\
\lambda V(\hat y)+H(\hat y,p-\eta\nabla\varphi(\hat y))&\ge0.
\end{align*}
Hence
\begin{align}\label{eq:disc_error_reverse}
\lambda[U(\hat x)-V(\hat y)]
&\le H(\hat y,p-\eta\nabla\varphi(\hat y))
-H(\hat x,p+\eta\nabla_h\varphi(\hat x))\nonumber\\
&\quad+\frac{d\nu}{\varepsilon}+\eta\nu\Delta_h\varphi(\hat x).
\end{align}

In both directions the momentum arguments are bounded by
$P_0=2(L+1)+\sqrt d$. Applying \eqref{eq:H_mixed_lipschitz} to
\eqref{eq:disc_error_forward} and \eqref{eq:disc_error_reverse}, we may take
\[
C_0=2(L+1)(L_c+L_fP_0),\qquad C_\eta=F(1+\sqrt d)+dN
\]
to bound either right-hand side by
$C_0\varepsilon+dNh/\varepsilon+C_\eta\eta$.
Define the localized diagonal errors
\[
M_\eta^+=\sup_x\{V(x)-U(x)-2\eta\varphi(x)\},\qquad
M_\eta^-=\sup_x\{U(x)-V(x)-2\eta\varphi(x)\}.
\]
Each is at most the corresponding doubled maximum, which in turn is at most
the value difference at its maximizing pair. Therefore
\[
\lambda M_\eta^\pm
\le C_0\varepsilon+\frac{dNh}{\varepsilon}+C_\eta\eta.
\]
Choose $\varepsilon=\sqrt h$ and let $\eta\downarrow0$.
For either sign, $M_\eta^\pm$ tends to the unpenalized supremum: the upper
bound is immediate, and the lower bound follows by evaluating at any fixed
point within an arbitrary positive tolerance of that supremum. We obtain
\[
\|V^h-V\|_\infty\le\frac{C_0+dN}{\lambda}\sqrt h,
\]
with a constant independent of $h$. No uniform Lipschitz estimate for $V^h$
is required.
\end{proof}

\begin{remark}[Sharpness of the discretization exponent]\label{rem:sqrt_h_sharpness}
Take $d=1$, $A=\{0\}$, $f=0$, $c(x,0)=\min\{|x|,1\}$, and
$\lambda=N=1$. These data satisfy our assumptions and $V(x)=c(x,0)$.
For $h=1/M$, $M\ge2$ an integer, the restriction $U_j=V^h(jh)$ satisfies
\[
(h+2)U_j-U_{j+1}-U_{j-1}=h\min\{|j|h,1\}.
\]
Set $\rho_h=(2+h-\sqrt{h^2+4h})/2$ and $D_h=\sqrt{1+4/h}$.
The summable kernel $G_j=\rho_h^{|j|}/D_h$ satisfies
$(h+2)G_j-G_{j+1}-G_{j-1}=h\delta_{j0}$.
Here $\delta_{j0}$ is the Kronecker delta. Convolution with the bounded cost
sequence therefore gives the unique bounded
solution. Since $(1-\rho_h)^2=h\rho_h$,
\[
V^h(0)=\frac{2}{D_h}\sum_{j\ge1}\rho_h^j\min\{jh,1\}
=\frac{2(1-\rho_h^M)}{D_h}.
\]
As $h=1/M\downarrow0$, $\rho_h^M\to0$ and $D_h\sqrt h\to2$. Thus
\[
\frac{|V^h(0)-V(0)|}{\sqrt h}\longrightarrow1.
\]
The $\sqrt h$ order cannot be improved to $o(\sqrt h)$ for the full bounded
Lipschitz data class, even with a singleton control set.
\end{remark}
}

\subsection{Total error decomposition and the \texorpdfstring{$nh$}{nh}-coupling}\label{sec:total_error}
In this section, we derive a unified error estimate that reveals the interaction
between the policy iteration error and the spatial discretization error.
The resolvent contraction bound deteriorates as the mesh is refined, increasing the number of iterations that this bound guarantees to be sufficient.

\paragraph{Unified error bound.}
Assume the hypotheses of both Theorems~\ref{thm:fixed_h_supnorm} and
\ref{thm:h_to_0}, including $\lambda>\Lip_x(f)$, with $N$ fixed
independently of $h$. Combining their estimates by the triangle inequality gives
\begin{equation}\label{eq:total_error_sharp}
\norm{V_n^h - V}_{L^\infty(\Rd)} \le C_1 \beta_h^n + C_2 \sqrt{h},
\end{equation}
where $C_1 = 2\lambda^{-1}\norm{c}_\infty$, $C_2 > 0$ is a constant independent of $h$ and $n$, and $\beta_h = \frac{2dN/h}{\lambda + 2dN/h} < 1$ is the contraction factor. This decomposition separates the iteration error and the discretization error, which are controlled by the resolvent estimate and consistency, respectively.

To better understand the asymptotic behavior as $h \to 0$, we observe that
\[
\beta_h = \left( 1 + \frac{\lambda h}{2dN} \right)^{-1}.
\]
{%
Using
\[
\log(1+a)\;\ge\;\frac{a}{1+a},
\qquad a>0,
\]
with $a=\lambda h/(2dN)$ gives
\[
\beta_h^{\,n}
=
\exp\!\Big[-n\log\Big(1+\frac{\lambda h}{2dN}\Big)\Big]
\le
\exp\!\Big(-\frac{\lambda\,nh}{2dN+\lambda h}\Big),
\]
and hence
\begin{equation}\label{eq:total_error_exp}
\norm{V_n^h - V}_{L^\infty(\Rd)}
\le
C_1 \exp\!\left( - \frac{\lambda\,nh}{2dN+\lambda h} \right) + C_2 \sqrt{h}.
\end{equation}
Since $h\in(0,1)$, one may also use the $h$-independent form
\[
\beta_h^{\,n}\le\exp\!\Big(-\frac{\lambda}{2dN+\lambda}\,nh\Big).
\]
}
The guaranteed iteration-error bound therefore decays as a function of $nh$.
This estimate does not identify the actual rate of policy iteration.

\paragraph{The $nh$-coupling and computational efficiency.}
The dependence on $nh$ in \eqref{eq:total_error_exp} gives the following
sufficient prescriptions for the iteration count:

\begin{itemize}
    \item \textbf{Mesh dependence of the sufficient bound.}
    For a fixed relative iteration tolerance, the resolvent estimate gives a
    sufficient count of order $h^{-1}$ as $h\downarrow0$.
    \item \textbf{A sufficient iteration count.}
    To make the iteration-error term in \eqref{eq:total_error_exp} of order
    $\sqrt h$, it is sufficient to choose
    \[
    n \ge {\frac{2dN+\lambda h}{2\lambda h}} \ln\left(\frac{1}{h}\right).
    \]
    Then $\|V_n^h-V\|_{L^\infty(\Rd)}\le(C_1+C_2)\sqrt h$.
    Thus an iteration count of order $h^{-1}\log(1/h)$ suffices; no matching
    lower bound on the required number of policy iterations is claimed.
\end{itemize}

\begin{remark}
The factor {$\frac{\lambda}{2dN+\lambda h}$} in the exponent quantifies the balance between
the discount parameter $\lambda$ and the artificial diffusion parameter $N$.
Increasing $N$ weakens the guaranteed resolvent contraction bound.
This estimate does not determine how the actual PI iteration count varies
with $N$.
\end{remark}

{
\begin{remark}[sharpness of the $nh$-coupling]\label{rem:nh_sharpness}
Both terms in \eqref{eq:total_error_sharp} are upper bounds, and the estimates
above are therefore sufficient conditions rather than descriptions of the
observed behavior.
The factor $\beta_h$ is the contraction constant of the \emph{affine} resolvent
map $T_\alpha$ associated with a frozen policy, and it does not use the
optimality of the improvement step; it is therefore natural that the observed
convergence should be faster.
The mesh refinement study of Section~\ref{subsec:numerics_1d_hsweep} quantifies
the gap for a fixed relative tolerance: its sufficient resolvent count is
$O(h^{-1})$, while the measured counts are much smaller over the tested range.
The $O(h^{-1}\log(1/h))$ prescription instead balances iteration error with
the $O(\sqrt h)$ discretization error. Neither experiment establishes an
asymptotic growth law for the observed iteration count.
The ratios measured in Sections~\ref{subsec:numerics_1d_fixed_h}
and~\ref{sec:num_2d} stay well below $\beta_h$ and eventually drop sharply
towards zero. Howard-type policy iteration is classically interpreted as a
Newton method applied to the Bellman equation, and superlinear or finite
convergence has been established in several settings, including monotone
discretizations of HJB equations~\cite{PutermanBrumelle1979,SantosRust2004,BokanowskiMarosoZidani2009}.
Proposition~\ref{prop:local_quadratic} gives a direct local quadratic estimate
under our current assumptions, without invoking those results. Its constant
scales as $h^{-2}$, so it does not imply a mesh-independent iteration count.
The $nh$-coupling describes a sufficient complexity bound
from the resolvent argument, not a matching lower bound on policy iteration.
\end{remark}
}

\section{Numerical experiments}\label{sec:exp}

{
The code for all experiments reported below, including scripts to regenerate
the figures and underlying metrics, is available at
\url{https://github.com/yeoneung/stationary_hj}.
Throughout this section, $\|\cdot\|_{L^\infty}$ denotes the maximum over grid
nodes and $\|\cdot\|_{2,h}$ the root-mean-square discrete norm
\begin{equation}\label{eq:rms_norm}
\|u\|_{2,h}
:=
\Bigl(\tfrac{1}{N_{\mathrm{grid}}}\textstyle\sum_j|u_j|^2\Bigr)^{1/2},
\end{equation}
where $N_{\mathrm{grid}}$ is the number of grid nodes.
Section~\ref{subsec:numerics_1d_fixed_h} isolates the fixed-$h$ policy-iteration
mechanism on a problem with an analytic solution,
Section~\ref{subsec:numerics_1d_hsweep} performs the mesh refinement study that
tests the vanishing-viscosity estimate and the $nh$-coupling, and
Section~\ref{sec:num_2d} treats a genuinely nonlinear two-dimensional benchmark.
}

\subsection{One-dimensional discounted quadratic control: fixed-\texorpdfstring{$h$}{h} PI convergence}
\label{subsec:numerics_1d_fixed_h}

We validate the semi-discrete policy iteration (PI) scheme on a one-dimensional
deterministic control problem with an analytic solution.
This experiment isolates the PI mechanism from discretization effects
and demonstrates the decay of the value iterates for fixed mesh size $h$.

\paragraph{Model problem.}
We consider the controlled dynamics
\begin{equation}
\dot x(t) = a(t),
{\qquad a(t)\in A:=[-a_{\max},a_{\max}],}
\end{equation}
with running cost
\begin{equation}
c(x,a) = \frac12 x^2 + \frac12 a^2,
\end{equation}
and discount factor $\lambda>0$.

Under the Hamiltonian convention
\[
H(x,p)=\sup_{a\in A}\{-c(x,a)-p a\},
\]
{%
the unconstrained maximizer is $a=-p$, so that
\begin{equation}\label{eq:1d_H}
H(x,p)
=
\begin{cases}
-\dfrac12x^2+\dfrac12p^2, & |p|\le a_{\max},\\[6pt]
-\dfrac12x^2-\dfrac12a_{\max}^2+a_{\max}|p|, & |p|> a_{\max},
\end{cases}
\end{equation}
the second branch corresponding to saturation at $\partial A$.
The function constructed below satisfies $|V'(x)|=P|x|\le PL<a_{\max}$ on
$[-L,L]$ by \eqref{eq:1d_amax}, so only the first branch of \eqref{eq:1d_H} is
active there and $V$ solves the reduced stationary HJB equation
\begin{equation}\label{eq:1d_hjb_reduced}
\lambda V(x) - \frac12 x^2 + \frac12 |V'(x)|^2 = 0,
\qquad x\in(-L,L).
\end{equation}
All statements below refer to \eqref{eq:1d_hjb_reduced} on the computational
interval, not to an equation on all of $\mathbb R$.
}

{
\begin{remark}[relation to Assumption~\ref{ass:A}]\label{rem:1d_assumptions}
The running cost $c(x,a)=\frac12x^2+\frac12a^2$ is not globally bounded on
$\mathbb R$, so this example does not satisfy \ref{A1} on the whole space.
Accordingly, the problem is solved on the truncated interval $[-L,L]$ with exact
Dirichlet data, and the experiment should be read as a \emph{bounded-domain
benchmark for the discrete policy-iteration mechanism}, not as an instance of the
whole-space theorems.
On $[-L,L]$ the relevant constants are finite, $\Lip_x(c)=L$ and
$\Lip_x(f)=0$, so the discount conditions in Lemma~\ref{lem:Vh_lipschitz} and
Theorem~\ref{thm:h_to_0} both hold trivially.
\end{remark}
}

\paragraph{Analytic solution.}
{
Inserting the ansatz $V(x)=\frac12Px^2$ into the stationary HJB equation gives
\[
\tfrac12\lambda Px^2-\tfrac12x^2+\tfrac12P^2x^2=0
\qquad\Longleftrightarrow\qquad
P^2+\lambda P-1=0 ,
\]
so that the problem admits the explicit quadratic solution
\begin{equation}\label{eq:1d_P}
V(x) = \frac12 P x^2,
\qquad
P = \frac{\sqrt{\lambda^2+4}-\lambda}{2}.
\end{equation}
(For $\lambda\downarrow0$ one recovers $P=1$, the undiscounted linear--quadratic
value.)
}
The optimal feedback control is
\begin{equation}
a_*(x) = - V'(x) = -P x.
\end{equation}

\paragraph{Spatial discretization.}
We truncate the domain to $[-L,L]$ with $L=3$
and discretize it uniformly with mesh size $h$:
\[
x_i = -L + i h, \qquad i=0,\dots,N_x.
\]

Dirichlet boundary conditions are imposed using the analytic solution:
\[
V^h(-L)=V(-L), \qquad V^h(L)=V(L).
\]

The discrete gradient and Laplacian are defined by centered differences:
\begin{align}
\nabla_h V_i &= \frac{V_{i+1}-V_{i-1}}{2h},\\
\Delta_h V_i &= \frac{V_{i+1}-2V_i+V_{i-1}}{h^2}.
\end{align}

{
\paragraph{Admissible control set and viscosity coefficient.}
The control set is the compact interval $A=[-a_{\max},a_{\max}]$ with
\begin{equation}\label{eq:1d_amax}
a_{\max}=1.2\,PL>PL,
\end{equation}
so that the continuous optimal feedback $a_*(x)=-Px$ is interior to $A$
on $[-L,L]$. For the reported meshes, we also verify directly that the
converged discrete feedback is strictly inside $A$, so clipping is inactive
at the discrete fixed point.
Since $f(x,a)=a$, condition \eqref{eq:Nchoice_new} reads
\begin{equation}\label{eq:1d_N}
N\ \ge\ \max\Bigl\{1,\frac{a_{\max}}{2}\Bigr\},
\end{equation}
and we take equality in \eqref{eq:1d_N}.
}

\paragraph{Semi-discrete PI scheme.}
Given a policy $a_n(x_i)$, the evaluation step solves
{
\begin{equation}\label{eq:1d_eval}
\lambda V_{n,i}^h
- \frac12 x_i^2
- \frac12 a_{n,i}^2
- a_{n,i} \nabla_h V_{n,i}^h
=
N h \Delta_h V_{n,i}^h,
\end{equation}
}
for $i=1,\dots,N_x-1$, {which is exactly $\mathcal L^h_{a_n}V^h_n=0$
for the operator \eqref{eq:Lalpha}.}

This is the discrete counterpart of
\[
\lambda V + \sup_a\{-c - p a\} = 0,
\]
with artificial viscosity ensuring monotonicity of the scheme.

The improvement step is {the greedy update}
\begin{equation}\label{eq:1d_improve}
a_{n+1,i} = {\Pi_{[-a_{\max},a_{\max}]}\bigl(-\nabla_h V_{n,i}^h\bigr)},
\end{equation}
{where $\Pi$ denotes the projection onto $A$.}

\paragraph{Numerical implementation.}
The policy evaluation step yields a tridiagonal linear system,
which is solved directly, up to roundoff, using the Thomas algorithm. This allows each PI step to be computed in linear time.
{%
The sign structure of \eqref{eq:1d_eval} is checked at every iteration: the
off-diagonal coefficients $-N/h\pm a_{n,i}/(2h)$ are nonpositive precisely under
\eqref{eq:1d_N}. This is the stencil monotonicity of
Lemma~\ref{lem:mono_operator}; together with the fixed Dirichlet data, it gives
the finite-grid comparison principle throughout the run.}

\paragraph{Experimental parameters.}
We use
\[
\lambda=1, \qquad L=3, \qquad h=0.03,
\]
and run PI for {$15$} iterations.
The initial policy is set to zero:
\[
a_0(x)\equiv 0.
\]
{%
With these values, \eqref{eq:1d_P}, \eqref{eq:1d_amax} and \eqref{eq:1d_N} give
\[
P=0.6180,
\qquad
a_{\max}=2.2249,
\qquad
N=1.1125,
\qquad
\beta_h=\frac{2N/h}{\lambda+2N/h}=0.98670 .
\]
}

\paragraph{Error metrics.}
At each iteration we compute{, in the norms of \eqref{eq:rms_norm},}
\begin{itemize}
\item {$\|V_n^h - V\|_{2,h}$ and $\|V_n^h - V\|_{L^\infty}$,}
\item {$\|V_n^h - V_{n-1}^h\|_{2,h}$ (PI residual).}
\end{itemize}

{
\paragraph{Predicted interior plateau.}
Because $V(x)=\frac12Px^2$ is quadratic, the centered operators are exact,
$\nabla_hV=V'$ and $\Delta_hV=V''=P$.
Substituting $V+C$ into \eqref{eq:HJBh_new} with the optimal feedback therefore
gives $\lambda C=NhP$, so that the shifted function
\begin{equation}\label{eq:1d_plateau}
\widehat V^h:=V+\frac{NhP}{\lambda}
\end{equation}
solves the semi-discrete equation exactly at every interior grid node of
$(-L,L)$, where $|V'(x)|=P|x|\le PL<a_{\max}$ so that the quadratic branch of
\eqref{eq:1d_H} is the active one.
It does \emph{not} satisfy the imposed Dirichlet conditions
$V^h(\pm L)=V(\pm L)$, so on the truncated interval $V^h$ differs from
$\widehat V^h$ by a boundary correction, and $NhP/\lambda$ must be read as the
interior plateau predicted by the consistency error, not as an exact identity
for the bounded-domain Dirichlet problem.
In the inactive-clipping regime, the correction $w:=V^h-\widehat V^h$
satisfies the nonlinear discrete equation
\[
\lambda w+Px\,\nabla_h w+\tfrac12|\nabla_h w|^2-Nh\,\Delta_h w=0,
\qquad w(\pm L)=-NhP/\lambda.
\]
This follows by subtracting the two Bellman equations; it is not the
homogeneous linear equation of a single frozen policy. We do not quantify the
width of this boundary correction here.
What matters for the experiment is the numerical observation that the correction
decays rapidly away from $x=\pm L$: at the center $x=0$, the correction is
below the displayed precision and $|V^h-V|$ agrees with its grid maximum to
that precision. The location of a floating-point maximizer within this nearly
flat plateau is not a stable diagnostic.
This is what makes the plateau a quantitative test rather than a qualitative one.
For the parameters above the prediction is $NhP/\lambda=2.0626\cdot10^{-2}$, and
the stored value of $\|V^h-V\|_{L^\infty}$ agrees with the prediction
to eight significant digits.
Note also that \eqref{eq:1d_plateau} deteriorates as $\lambda\downarrow0$,
in agreement with the loss of coercivity discussed after \eqref{eq:HJB}.
}

\paragraph{Results.}
Figure~\ref{fig:1d_fixedh_main} illustrates the numerical behavior of the proposed scheme. In particular, the error curve in Figure~\ref{fig:1d_fixedh_main}(b) exhibits two distinct regimes.

For small values of $n$, the error $\|V_n^h - V\|$ decays rapidly,
{consistently with the geometric bound of
Theorem~\ref{thm:fixed_h_supnorm} for fixed mesh size $h$.}
As $n$ increases, the decay slows down and the error eventually reaches a plateau
determined by the discretization error $\|V^h - V\|$.
Beyond this point, further iterations yield negligible improvement.
{%
The dashed line in Figure~\ref{fig:1d_fixedh_main}(b) is the level
interior plateau $NhP/\lambda$ predicted by \eqref{eq:1d_plateau}; the computed
error settles on it after six iterations.}

This behavior clearly separates the iteration error from the discretization error,
in agreement with the estimate
\[
\|V_n^h - V\|_\infty
\le
C_1 e^{-c n h} + C_2 \sqrt{h}.
\]
Moreover, the residual decay shown in Figure~\ref{fig:1d_fixedh_main}(c)
confirms {that the PI residual contracts to roundoff level within a
few iterations.}

{
\begin{remark}[observed rate versus the contraction factor $\beta_h$]
\label{rem:rate_vs_beta}
Theorem~\ref{thm:fixed_h_supnorm} is an upper bound, and for this example it is
far from being attained.
With $\beta_h=0.98670$, the bound
$\|V_n^h-V^h\|_{L^\infty}\le\beta_h^{\,n}\|V_0^h-V^h\|_{L^\infty}$ requires
$n=1032$ iterations to guarantee six digits, whereas the computed sequence
achieves this in $7$ iterations; the measured ratios
$\|V^h_{n+1}-V^h\|_{L^\infty}/\|V^h_n-V^h\|_{L^\infty}$ for $n=0,\dots,4$ are
$4.6\cdot10^{-1}$, $6.3\cdot10^{-1}$, $5.0\cdot10^{-1}$, $3.3\cdot10^{-1}$,
$4.2\cdot10^{-2}$, {i.e.\ far below $\beta_h$ throughout and dropping
sharply in the last step, rather than being governed by a fixed factor}.
This is expected: $\beta_h$ is the contraction factor of the \emph{affine}
resolvent map $T_\alpha$ and does not use the optimality of the improvement
step, whereas the iteration actually performed solves the frozen-policy equation
exactly at each stage.
The late decrease in these ratios is consistent with the classical interpretation of
Howard-type policy iteration as a Newton method for the Bellman
equation~\cite{PutermanBrumelle1979,SantosRust2004,BokanowskiMarosoZidani2009},
and with the direct fixed-$h$ quadratic estimate of
Proposition~\ref{prop:local_quadratic}; see also Remark~\ref{rem:nh_sharpness}.
The experiments therefore confirm the validity of the bound, and show that the
practical convergence on these benchmarks is considerably faster than the
guaranteed rate.
\end{remark}

\begin{remark}[monotonicity of the PI sequence]\label{rem:mono_numeric}
The pointwise monotonicity $V^h_{n+1}\le V^h_n$ predicted by
Proposition~\ref{prop:mono} is monitored at every iteration through
$\max_i(V^h_{n+1,i}-V^h_{n,i})$.
Over the whole run this quantity never exceeds $4\cdot10^{-15}$, i.e.\ it is zero
up to roundoff.
\end{remark}
}

\begin{figure}[!t]
  \centering
  \begin{subfigure}[t]{0.32\textwidth}
    \centering
    \includegraphics[width=\textwidth]{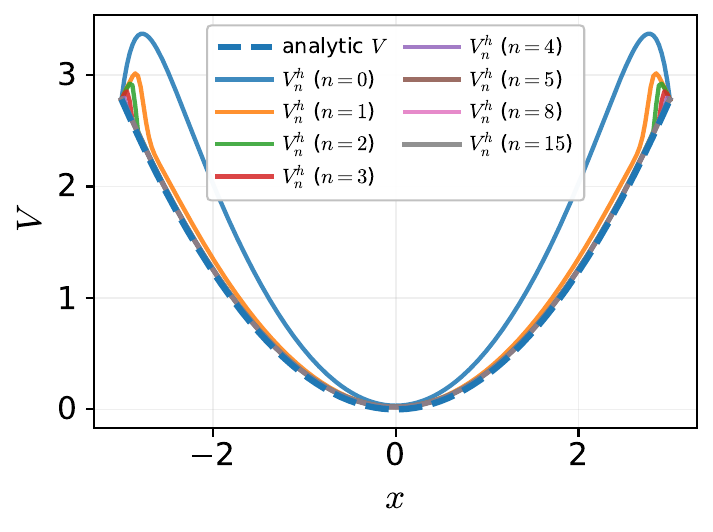}
    \caption{Value iterates vs.\ analytic solution.}
    \label{fig:1d_pi_curves}
  \end{subfigure}\hfill
   \begin{subfigure}[t]{0.32\textwidth}
   \centering
  \includegraphics[width=\textwidth]{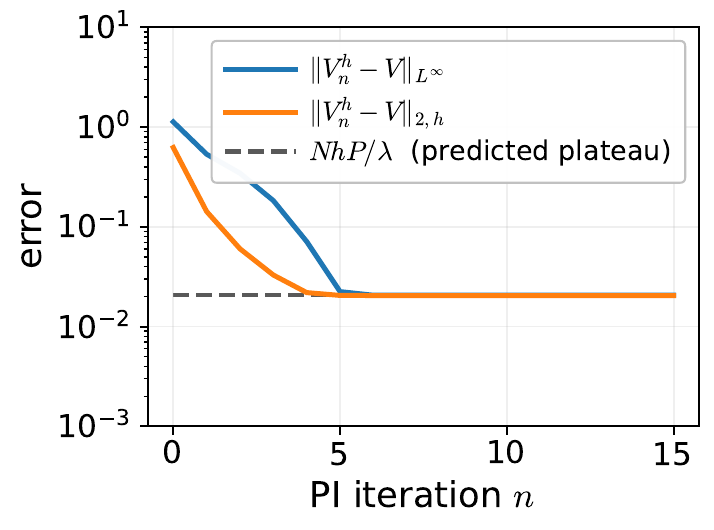}
    \caption{Error to the analytic solution.}
  \end{subfigure}\hfill
  \begin{subfigure}[t]{0.32\textwidth}
  \centering
    \includegraphics[width=\textwidth]{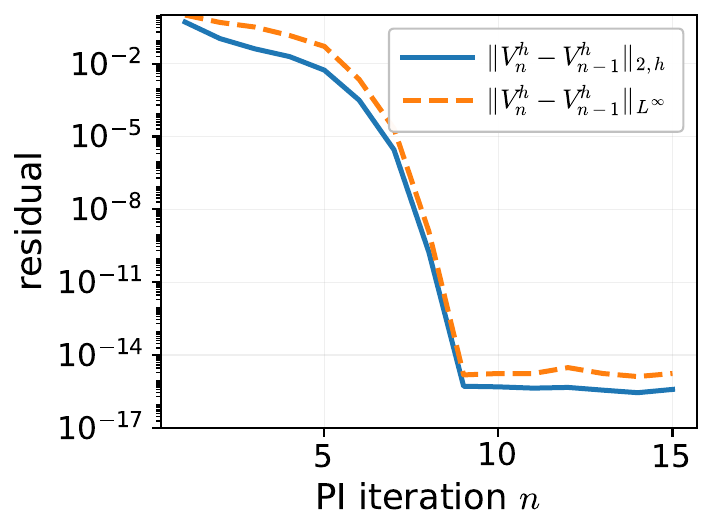}
    \caption{Policy iteration residual.}
    \label{fig:1d_pi_residual}
  \end{subfigure}

  \caption{Fixed-$h$ policy iteration for the one-dimensional discounted quadratic control problem. The left panel shows convergence of the value iterates, the middle panel illustrates the decay-then-plateau behavior of the error, and the right confirms {the collapse} of the PI residual.
  Together, these results clearly separate iteration error from discretization error{; the dashed line in the middle panel is the interior plateau $NhP/\lambda$ predicted by \eqref{eq:1d_plateau}}.}
  \label{fig:1d_fixedh_main}
\end{figure}

{
\subsection{Mesh refinement: the vanishing-viscosity rate and the \texorpdfstring{$nh$}{nh}-coupling}
\label{subsec:numerics_1d_hsweep}

The coefficients of the one-dimensional problem do not depend on $h$, so this is
the benchmark on which the two mesh-dependent predictions of the theory can
actually be tested, namely the vanishing-viscosity estimate of
Theorem~\ref{thm:h_to_0} and the $nh$-coupling of
Section~\ref{sec:total_error}.
We repeat the experiment of Section~\ref{subsec:numerics_1d_fixed_h} for
$h\in\{0.12,0.06,0.03,0.015,0.0075\}$ with $\lambda=1$, $L=3$, $a_0\equiv0$,
keeping $a_{\max}$ and $N=1.1125$ fixed.
For each mesh we record the converged discretization error
$\|V^h-V\|_{L^\infty}$ and the number
\[
n_{\rm obs}
:=
\min\bigl\{n:\ \|V^h_n-V^h\|_{L^\infty}\le10^{-6}\|V^h_0-V^h\|_{L^\infty}\bigr\},
\]
together with the sufficient integer count
$n_{\beta}=\lceil\log(10^{-6})/\log\beta_h\rceil$ predicted by
Theorem~\ref{thm:fixed_h_supnorm}.

\begin{table}[!ht]
\centering
\begin{tabular}{cccccc}
\hline
$h$ & $\beta_h$ & $n_{\rm obs}$ & $n_{\beta}$
& $\|V^h-V\|_{L^\infty}$ & $\|V^h-V\|_{L^\infty}/h$\\
\hline
$0.12$   & $0.94883$ & $4$  & $264$  & $8.2505\cdot10^{-2}$ & $0.68754$\\
$0.06$   & $0.97374$ & $5$  & $520$  & $4.1252\cdot10^{-2}$ & $0.68754$\\
$0.03$   & $0.98670$ & $7$  & $1032$ & $2.0626\cdot10^{-2}$ & $0.68754$\\
$0.015$  & $0.99330$ & $9$  & $2057$ & $1.0313\cdot10^{-2}$ & $0.68754$\\
$0.0075$ & $0.99664$ & $13$ & $4106$ & $5.1565\cdot10^{-3}$ & $0.68754$\\
\hline
\end{tabular}
\caption{Mesh refinement for the one-dimensional benchmark. The last column is
constant to the displayed precision and equal to $NP/\lambda$, in agreement with
the interior plateau \eqref{eq:1d_plateau}.}
\label{tab:1d_hsweep}
\end{table}

Two conclusions follow from Table~\ref{tab:1d_hsweep} and
Figure~\ref{fig:1d_hsweep}.

First, the measured discretization error agrees with $NPh/\lambda$ to the
displayed precision on every mesh, so that its observed order is $1$.
This is consistent with, but strictly better than, the $O(\sqrt h)$ bound of
Theorem~\ref{thm:h_to_0}: the $\sqrt h$ rate is a worst-case estimate obtained by
doubling of variables for merely Lipschitz viscosity solutions, whereas the value
function of the present example is smooth, and for smooth $V$ the consistency
error of the scheme is $O(h)$.
Remark~\ref{rem:sqrt_h_sharpness} establishes sharpness analytically with a
bounded nonsmooth value function. The present mesh experiment instead measures
the behavior of a smooth benchmark.

Second, for a fixed relative tolerance $\varepsilon=10^{-6}$,
\[
n_\beta=\left\lceil\frac{\log(1/\varepsilon)}
{\log(1+\lambda h/(2dN))}\right\rceil=O(h^{-1}).
\]
Halving $h$ four times increases $n_{\rm obs}$ from $4$ to $13$, while
$n_\beta$ grows by a factor of about $16$. These five meshes show a substantial
gap between observed counts and the sufficient bound, but do not establish an
asymptotic law for $n_{\rm obs}$.
The different prescription $n=O(h^{-1}\log(1/h))$ in
Section~\ref{sec:total_error} balances iteration error with a mesh-dependent
$O(\sqrt h)$ discretization error; it is not the fixed-tolerance prediction.

\begin{figure}[!t]
  \centering
  \begin{subfigure}[t]{0.48\textwidth}
    \centering
    \includegraphics[width=\textwidth]{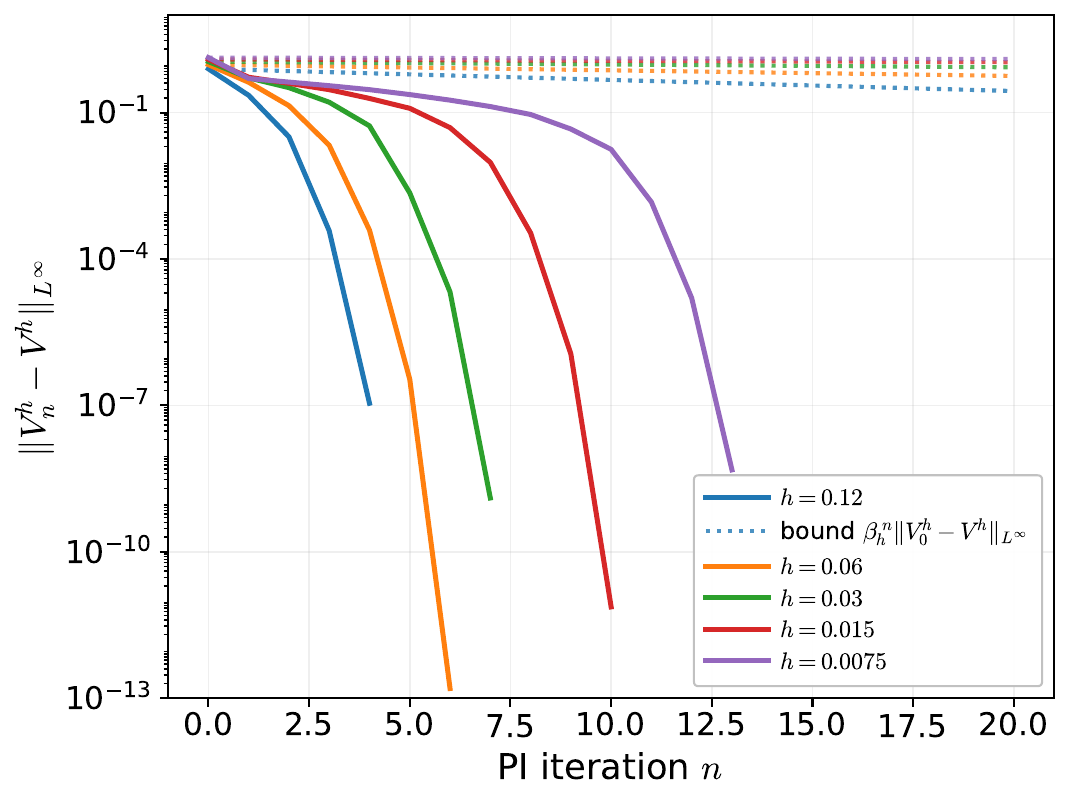}
    \caption{$\|V_n^h-V^h\|_{L^\infty}$ for five mesh sizes (solid) against the
    bounds $\beta_h^{\,n}\|V_0^h-V^h\|_{L^\infty}$ (dotted).}
    \label{fig:1d_hsweep_error}
  \end{subfigure}\hfill
  \begin{subfigure}[t]{0.48\textwidth}
    \centering
    \includegraphics[width=\textwidth]{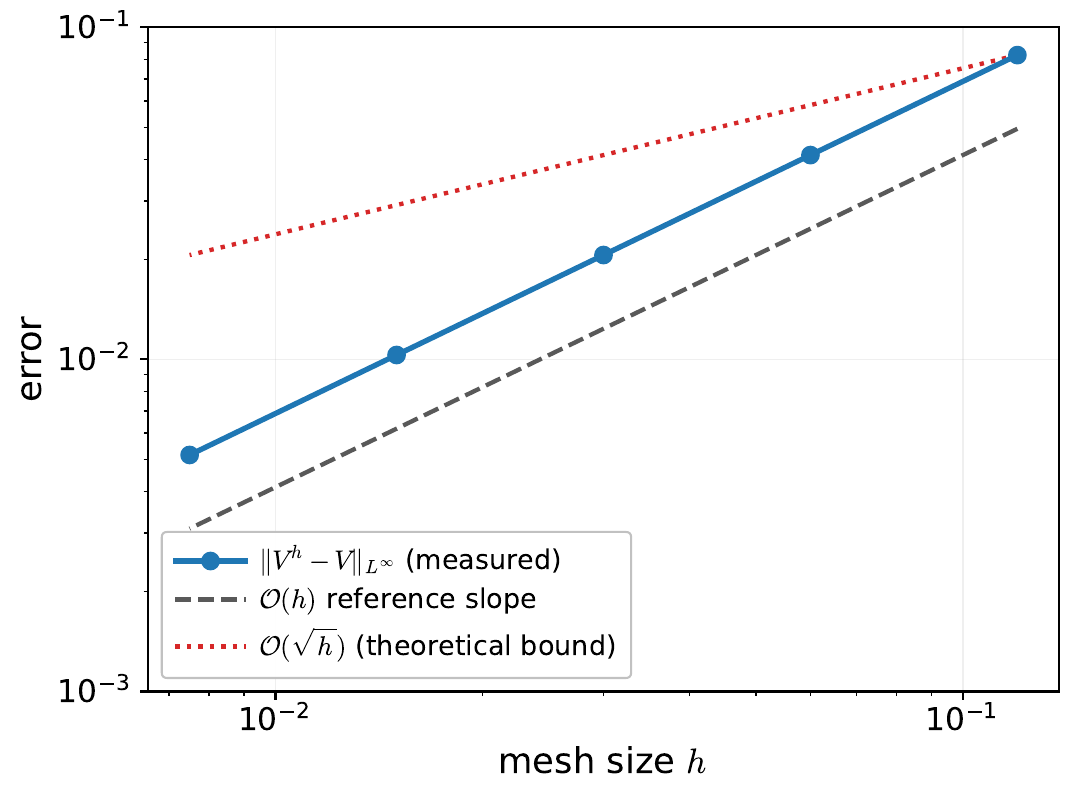}
    \caption{Converged discretization error $\|V^h-V\|_{L^\infty}$ versus $h$,
    with $\mathcal O(h)$ and $\mathcal O(\sqrt h)$ reference slopes.}
    \label{fig:1d_hsweep_plateau}
  \end{subfigure}
  \caption{Mesh refinement study for the one-dimensional benchmark. The left
  panel shows that the guaranteed rate $\beta_h$ is valid but far from sharp;
  the right panel shows that the observed vanishing-viscosity rate is $O(h)$ for
  this smooth example.}
  \label{fig:1d_hsweep}
\end{figure}
}

\subsection{Nonlinear two-dimensional benchmark: fixed-\texorpdfstring{$h$}{h} PI convergence}
\label{sec:num_2d}

We next validate the semi-discrete policy iteration (PI) scheme
in a genuinely nonlinear two-dimensional setting.
As in the one-dimensional fixed-$h$ experiment,
the purpose is to isolate the PI mechanism and examine the decay
of the iterates with respect to the iteration index $n$
for a fixed mesh size $h$.
In the present benchmark, an exact \emph{discrete} reference solution
is manufactured, so that the convergence behavior of the PI iterates
can be observed without an additional continuous--discrete mismatch.

\paragraph{Model problem.}
We consider the two-dimensional deterministic control problem
\begin{equation}
\dot z(t) = b(z(t)) + a(t),
\qquad
z(t)=(x(t),y(t))\in\mathbb{R}^2,
\qquad
a(t)=(a_x(t),a_y(t))\in{[-a_{\max},a_{\max}]^2},
\end{equation}
with drift field
\begin{align}
b_1(x,y)
&=
0.28\sin x
+
0.14\tanh(0.80\,y)
+
0.06\cos(1.20\,x-0.40\,y),
\\
b_2(x,y)
&=
-0.24\sin y
+
0.12\tanh(0.70\,x)
-
0.05\sin(0.90\,x+0.80\,y).
\end{align}
The running cost is taken as
\begin{equation}
c_h(z,a)=q_h(z)+\frac12|a|^2,
\end{equation}
and the infinite-horizon discounted objective is
\begin{equation}
J_h(z;a)
=
\int_0^\infty e^{-\lambda t}
\left(
q_h(z(t))+\frac12|a(t)|^2
\right)\,dt,
\qquad \lambda>0.
\end{equation}

\paragraph{Manufactured reference solution.}
To obtain a nonlinear benchmark with known ground truth,
we prescribe the smooth nonseparable function
\begin{align}
V^\ast(x,y)
&=
0.08(x^2+1.40\,y^2)
+
0.11\sin(1.30\,x+0.20)\cos(0.70\,y-0.10)
\nonumber\\
&\quad
+
0.055\tanh(0.90\,xy)
+
0.045\sin(0.60\,xy+0.35\,x-0.25\,y)
\nonumber\\
&\quad
+
0.035\cos(1.70\,x-0.40\,y)
+
0.025\arctan(0.80\,x-1.10\,y)
\nonumber\\
&\quad
+
0.020\sin(2.20\,x)\sin(1.40\,y).
\end{align}
This reference profile is intentionally nonlinear and nonseparable,
so that the resulting benchmark is substantially less structured
than the quadratic examples considered earlier.

Rather than requiring $V^\ast$ to solve the continuous HJB equation,
we construct the source term so that $V^\ast$ is an exact solution of the semi-discrete scheme for the fixed mesh size $h$.

Let $\nabla_h$ and $\Delta_h$ denote the centered discrete gradient
and Laplacian.
We then define
\begin{equation}
q_h(z)
=
\lambda V^\ast(z)
-
b(z)\cdot \nabla_h V^\ast(z)
+
\frac12 \bigl|\nabla_h V^\ast(z)\bigr|^2
-
N h \Delta_h V^\ast(z).
\end{equation}
With this choice, $V^\ast$ is an exact fixed point of the
\emph{discrete} semi-discrete scheme, and the corresponding discrete
reference feedback is
\begin{equation}
a_h^\ast(z)=-\nabla_h V^\ast(z).
\end{equation}
Thus, in contrast to the 1D analytic benchmark,
the role of the ground truth is played here by a manufactured exact
reference for the fixed-$h$ discrete problem.

{
\begin{remark}[scope of the manufactured benchmark]\label{rem:2d_scope}
As in Remark~\ref{rem:1d_assumptions}, this example is not a literal instance of
the whole-space theorems: the manufactured cost $q_h$ is not bounded on
$\mathbb R^2$, and the problem is solved on $[-L,L]^2$ with exact Dirichlet data.
What is verified is the bounded-domain analogue of the fixed-$h$ monotonicity
and contraction argument, and all constants entering the analysis are finite on
the computational box and are reported below.

The source $q_h$ depends explicitly on $h$, so this construction defines a
\emph{different} discrete problem for each mesh size, and $V^\ast$ plays the role
of $V^h$, not of the continuous value function $V$.
Consequently the experiment below tests the bounded-domain counterpart of the
fixed-$h$ argument in Theorem~\ref{thm:fixed_h_supnorm}, i.e.\ the convergence
$V^h_n\to V^h$ for a fixed mesh, and carries no information about the
vanishing-viscosity estimate of Theorem~\ref{thm:h_to_0} or about the
$nh$-coupling; those are tested in Section~\ref{subsec:numerics_1d_hsweep}, where
the coefficients are independent of $h$.
All error quantities reported below are therefore pure policy-iteration errors,
with zero discretization component.
\end{remark}
}

\paragraph{Spatial truncation and discretization.}
Although the state space is $\mathbb{R}^2$,
we truncate the computational domain to
\[
[-L,L]^2,
\qquad L=2,
\]
with a uniform Cartesian mesh of size $h$:
\[
x_i=-L+ih,\qquad y_j=-L+jh.
\]
Dirichlet boundary conditions are imposed directly from the reference
solution,
\[
V^h(x,y)=V^\ast(x,y)
\qquad \text{for } (x,y)\in \partial([-L,L]^2),
\]
so that boundary effects do not contaminate the interior PI dynamics.

The discrete gradient and Laplacian are defined by centered differences:
\begin{align}
\nabla_h V_{i,j}
&=
\left(
\frac{V_{i+1,j}-V_{i-1,j}}{2h},
\frac{V_{i,j+1}-V_{i,j-1}}{2h}
\right),\\
\Delta_h V_{i,j}
&=
\frac{
V_{i+1,j}+V_{i-1,j}+V_{i,j+1}+V_{i,j-1}
-4V_{i,j}
}{h^2}.
\end{align}

\paragraph{Semi-discrete PI scheme.}
Given a policy $a_n=(a_{x,n},a_{y,n})$,
the evaluation step solves the linear difference equation
{
\begin{equation}\label{eq:2d_exp}
\lambda V_{n}^{h}
-
q_h
-
\frac12 |a_n|^2
-
(b+a_n)\cdot \nabla_h V_n^{h}
=
N h \Delta_h V_n^{h},
\end{equation}
that is, $\mathcal L^h_{a_n}V^h_n=0$ for the operator \eqref{eq:Lalpha} with
$f_{a_n}=b+a_n$ and $c_{a_n}=q_h+\frac12|a_n|^2$.
}
The artificial viscosity coefficient $N$
is chosen to satisfy the monotonicity requirement for the clipped drift.
{%
Since the stencil is monotone as soon as the \emph{componentwise} condition
$N\ge\frac12\max_i\|b_i+a_i\|_{L^\infty}$ holds, and since
\eqref{eq:Nchoice_new} additionally normalizes $N\ge1$, we take
\begin{equation}\label{eq:2d_N}
N
=
\max\Bigl\{1,\
1.05\cdot\tfrac12\bigl(\max_i\|b_i\|_{L^\infty}+a_{\max}\bigr)\Bigr\}.
\end{equation}
}

{
\paragraph{Admissible control set and reported quantities.}
The control box is $A=[-a_{\max},a_{\max}]^2$ with
\begin{equation}\label{eq:2d_amax}
a_{\max}=1.10\max_{z,i}\bigl|(\nabla_hV^\ast(z))_i\bigr|,
\end{equation}
so that the reference feedback $a_h^\ast=-\nabla_hV^\ast$ never activates the
clipping and is an admissible fixed point of the improvement step.
For $L=2$ and $h=0.05$ the computed values are
\[
\max_{z,i}|(\nabla_hV^\ast(z))_i|=0.5522,
\quad
a_{\max}=0.6074,
\quad
\max_i\|b_i\|_{L^\infty}=0.4432,
\quad
N=1.0000,
\]
Thus \eqref{eq:2d_N} holds, since the monotonicity threshold is
\[
\tfrac12\bigl(\max_i\|b_i\|_{L^\infty}+a_{\max}\bigr)=0.5253<N.
\]
The remaining grid-based quantities used in the comparison are
\[
\|q_h\|_{L^\infty}=0.9321,
\qquad
C_1:=\frac{2(\|q_h\|_{L^\infty}+a_{\max}^2)}{\lambda}=2.6022,
\qquad
\beta_h=\frac{2dN/h}{\lambda+2dN/h}=0.98765,
\]
An analytic bound for the drift derivatives is
\[
|Db(x,y)|\le
\begin{pmatrix}0.352&0.136\\0.129&0.280\end{pmatrix}
\quad\text{entrywise on }\mathbb R^2.
\]
Thus $\Lip_x(f)=\Lip(b)\le
\sqrt{0.352^2+0.136^2+0.129^2+0.280^2}=0.48728$ (rounded), and the
stronger discount condition of Lemma~\ref{lem:Vh_lipschitz}, which also
implies the condition of Theorem~\ref{thm:h_to_0}, is certified by
$\sqrt d\,\Lip_x(f)\le0.68912<1=\lambda$.
The displayed norms of $q_h$ and $b_i$ above are grid maxima, not certified
whole-space suprema.
}

The policy improvement step is {the greedy update analysed in
Theorem~\ref{thm:fixed_h_supnorm},}
\begin{equation}\label{eq:2d_improve}
a_{n+1}
=
\Pi_{[-a_{\max},a_{\max}]^2}\bigl(-\nabla_h V_n^h\bigr),
\end{equation}
{where $\Pi$ denotes componentwise clipping onto the admissible
control box.
A damped variant of \eqref{eq:2d_improve} is discussed separately in
Remark~\ref{rem:relaxed}.}

\paragraph{Linear solver.}
{%
Each policy evaluation step yields a monotone linear system on the fixed grid.
Because Theorem~\ref{thm:fixed_h_supnorm} assumes \emph{exact} policy
evaluation, we solve this system by a sparse direct factorization, so that the
error floor of the experiment is determined by roundoff rather than by an inner
iterative tolerance.
A vectorized red--black Gauss--Seidel SOR solver is also provided in the
accompanying code and produces indistinguishable results down to its stopping
tolerance.
}

\paragraph{Experimental parameters.}
Unless otherwise specified, we fix
\[
\lambda=1,
\qquad
L=2,
\qquad
h=0.05.
\]
Policy iteration is run for {$15$} iterations.
The initial policy is chosen deliberately far from the reference one,
roughly as the opposite of $a_h^\ast$ plus a smooth oscillatory
perturbation, so that the successive PI iterates remain visually
distinguishable during the early stages of the experiment.

\paragraph{Error metrics.}
We use the same error metrics as in the one-dimensional experiment. In addition, one-dimensional slices of the value function are recorded for representative fixed values of $x$ and $y$.

\begin{figure}[!t]
  \centering
  \begin{subfigure}[t]{0.32\textwidth}
    \centering
    \includegraphics[width=\textwidth]{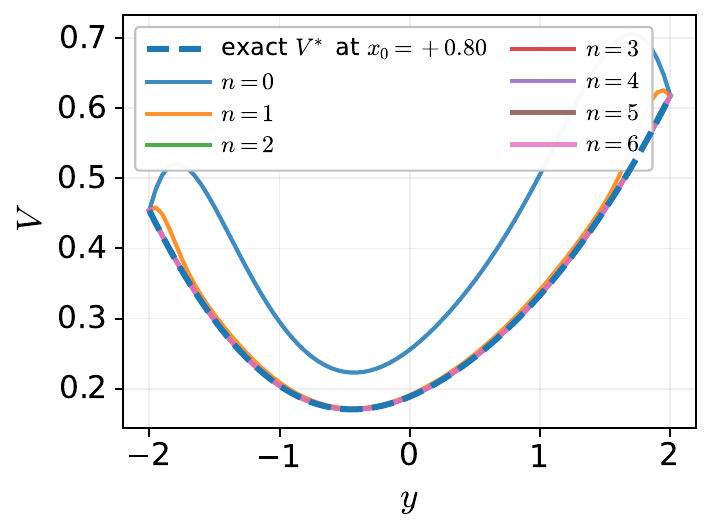}
    \caption{PI iterates $V_n^h(x_0,\cdot)$ versus the manufactured reference solution $V^\ast(x_0,\cdot)$ at the fixed slice $x_0=0.80$.}
    \label{fig:2d_slice_x}
  \end{subfigure}\hfill
  \begin{subfigure}[t]{0.32\textwidth}
    \centering
    \includegraphics[width=\textwidth]{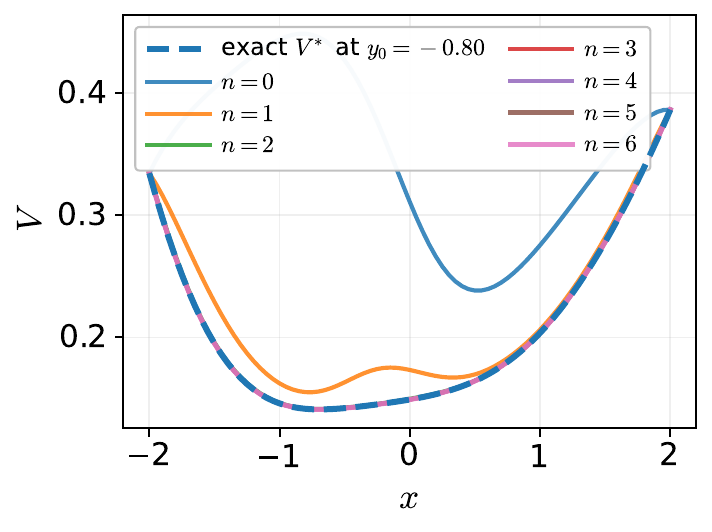}
    \caption{PI iterates $V_n^h(\cdot,y_0)$ versus the manufactured reference solution $V^\ast(\cdot,y_0)$ at the fixed slice $y_0=-0.80$.}
    \label{fig:2d_slice_y}
  \end{subfigure}\hfill
  \begin{subfigure}[t]{0.32\textwidth}
    \centering
    \includegraphics[width=\textwidth]{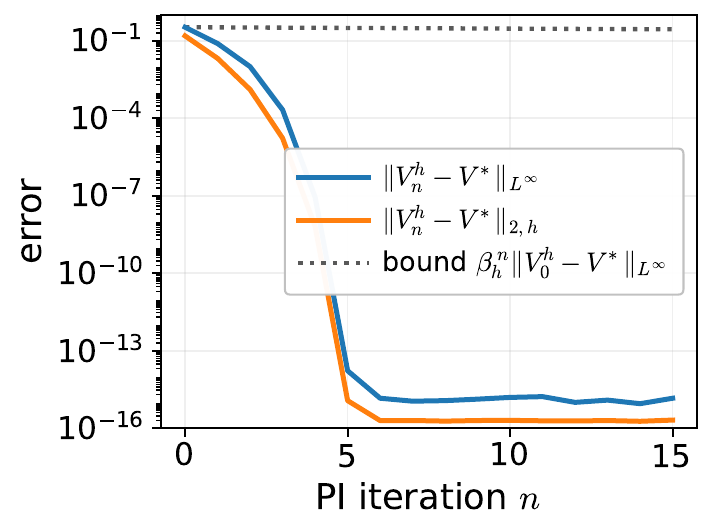}
    \caption{Error to the manufactured reference solution versus PI index $n$ (log scale){, together with the bound $\beta_h^{\,n}\|V_0^h-V^\ast\|_{L^\infty}$}.}
    \label{fig:2d_error}
  \end{subfigure}
  \caption{Fixed-$h$ PI convergence in the nonlinear 2D benchmark{\ with the greedy update \eqref{eq:2d_improve}}. The three panels show two representative one-dimensional slices of the value iterates and the decay of the global error to the reference solution.}
  \label{fig:2d_nonlinear_main}
\end{figure}

Figure~\ref{fig:2d_nonlinear_main} shows that the value iterates decrease
monotonically with respect to the PI index $n$.
The one-dimensional slices make this behavior visible pointwise along
representative coordinate directions, while the global error curves
quantify the decay of the grid maximum $\|V_n^h-V^\ast\|_{L^\infty}$ and
the RMS error $\|V_n^h-V^\ast\|_{2,h}$ for fixed $h$.
Because the benchmark is discrete-manufactured, the reference function
$V^\ast$ is an exact fixed point of the discrete scheme, and therefore
the error curves decay toward zero
{up to roundoff: $\|V^h_n-V^\ast\|_{L^\infty}$ drops from
$3.36\cdot10^{-1}$ to $1.4\cdot10^{-15}$ in six iterations, with successive
ratios $2.3\cdot10^{-1}$, $1.3\cdot10^{-1}$, $2.1\cdot10^{-2}$,
$4.1\cdot10^{-4}$, $1.9\cdot10^{-7}$.
As in Remark~\ref{rem:rate_vs_beta}, this {empirically Newton-like
local behavior on the present benchmark} is much faster than the guaranteed
factor $\beta_h=0.98765$, which is drawn as a dotted line in
Figure~\ref{fig:2d_error}.
The monotonicity diagnostic $\max_{i,j}(V^h_{n+1,i,j}-V^h_{n,i,j})$ never becomes
positive beyond roundoff, in agreement with Proposition~\ref{prop:mono}.}

{
\begin{remark}[a damped variant of the improvement step]\label{rem:relaxed}
It is also instructive to run the relaxed update
\begin{equation}\label{eq:2d_relaxed}
a_{n+1}
=
(1-\theta)a_n
+
\theta\,\Pi_{[-a_{\max},a_{\max}]^2}\bigl(-\nabla_hV_n^h\bigr),
\qquad\theta\in(0,1],
\end{equation}
which for $\theta<1$ makes the approach to the reference solution visually
clearer.
We stress that \eqref{eq:2d_relaxed} is \emph{not} the algorithm analysed in this
paper: the policy-improvement identity of Lemma~\ref{lem:fixed_point_representation},
on which the proof of Theorem~\ref{thm:fixed_h_supnorm} rests, uses the exact
maximizer and fails for $\theta<1$.
Run with $\theta=0.18$ and $60$ iterations, the damped scheme still converges
monotonically to $V^\ast$. Its error ratio is about $0.75$ in the first few
steps and approaches $0.6724=(1-0.18)^2$ at later steps before roundoff;
this observed ratio depends on the relaxation and is not the bound
$\beta_h=0.98765$.
The corresponding figures are provided with the accompanying code and are
reported only as a supplementary damped variant.
\end{remark}
}

{
\subsection{Boundary-free neural policy evaluation}
\label{sec:num_inexact}

Sections~\ref{subsec:numerics_1d_fixed_h}--\ref{sec:num_2d} illustrate
the semi-discrete mechanism on bounded computational domains with prescribed
Dirichlet data, using direct linear solves for policy evaluation up to roundoff.
The present diagnostic uses the same interior benchmark coefficients as
Section~\ref{sec:num_2d}, including $a_{\max}=0.6074$, $N=1.0000$, and $q_h$,
but changes both the evaluator and the boundary treatment.
It therefore does not isolate the effect of inexact policy evaluation.

We fit a neural network $V_\vartheta$ by minimizing the mean square of the
interior semi-discrete residual
\begin{equation}\label{eq:pinn_residual}
r_n=\lambda V_\vartheta-q_h-\tfrac12|a_n|^2
-(b+a_n)\cdot\nabla_hV_\vartheta-Nh\,\Delta_hV_\vartheta
\end{equation}
over the interior nodes, in the spirit of physics-informed neural
networks~\cite{raissi2019physics}, with no boundary supervision.
Here $\nabla_h$ and $\Delta_h$ act on network outputs at shifted grid nodes;
they are not automatic derivatives. Thus the fit targets the semi-discrete
interior equation, not the underlying continuous PDE.
The boundary values remain free unknowns, so the residual equations alone
do not determine a unique solution. The network parameterization,
initialization, and optimization influence which fit is obtained; we assert
no uniqueness of that selection.

\paragraph{Implementation and reproducibility.}
The network has five hidden layers of width $128$, $\tanh$ activations,
and input scaling $(x,y)\mapsto(x/L,y/L)$. Each outer evaluation uses $2000$
full-batch Adam steps with learning rate $10^{-3}$ and weight decay $10^{-10}$,
followed by at most $100$ L-BFGS iterations (step size $1$, strong-Wolfe line
search, history size $50$, and gradient/change tolerances $10^{-14}$).
Network weights are retained between outer iterations; both optimizers are
reset. Each of the two runs starts from seed $1234$ and uses the initial policy
of Section~\ref{sec:num_2d}. The reported rerun uses CPU float64, two PyTorch
threads, Python $3.11.5$, NumPy $1.26.4$, and PyTorch $2.5.1$.
The accompanying code saves the evaluated grid arrays, metrics, and environment
metadata, and regenerates the figures from those arrays without retraining.
Optimizer trajectories and the last reported digits may vary across environments.

\paragraph{Observed monotonicity.}
For neural iterates define the signed maximum increment
\[
m_n=\max_{z\in G_h}(V_\vartheta^n(z)-V_\vartheta^{n-1}(z)),
\]
where $G_h$ includes all nodes, including the boundary. Interior and boundary
maxima are also saved separately; $m_n$ is a signed maximum, not its positive part.
With the greedy update \eqref{eq:2d_improve}, $m_n>10^{-12}$ at $9$ of the
$10$ updates and its largest value is $2.35\cdot10^{-2}$
(the largest interior increment is $1.93\cdot10^{-2}$).
With the damped update \eqref{eq:2d_relaxed}, $\theta=0.18$, there is no
violation over these $10$ updates; see Figure~\ref{fig:pinn2d_mono}.
This is an empirical observation on one benchmark, not a monotonicity guarantee
for damping. Neither Proposition~\ref{prop:mono} nor
Theorem~\ref{thm:fixed_h_supnorm} applies to these boundary-free approximate
evaluations, and this experiment cannot establish that exact evaluation is
necessary for monotonicity. A comparison with an approximate evaluator using
the same Dirichlet data would be needed to isolate evaluation error.

For a fixed policy $a$, a precise distinction is available. Let $V_a$ solve
the discrete evaluation equation on the finite grid with boundary values $g$,
and let $\widetilde V$ have interior residual $r$. The discrete maximum
principle gives
\begin{equation}\label{eq:inexact_pi}
\|\widetilde V-V_a\|_{\infty,G_h}
\le \max\left\{\lambda^{-1}\|r\|_{\infty,G_h^\circ},
\|\widetilde V-g\|_{\infty,\partial G_h}\right\}.
\end{equation}
Indeed, at a positive interior maximum of the error the off-diagonal weights
give $\lambda\max(\widetilde V-V_a)\le\|r\|_\infty$, and the same argument
applies to its negative. This fixed-policy estimate includes boundary error;
it is not an inexact-PI contraction theorem.

\paragraph{Accuracy and boundary diagnostic.}
For $\theta=0.18$, the grid maximum error falls from $6.46\cdot10^{-1}$ at
$n=0$ to $1.59\cdot10^{-2}$ at $n=8$, then rises to $2.21\cdot10^{-2}$ at
$n=10$, about $2.66\%$ of $\|V^\ast\|_{L^\infty}=0.8312$.
The RMS error decreases throughout this run, from $3.32\cdot10^{-1}$ to
$3.31\cdot10^{-3}$. The interior maximum error also rises at the last two
steps, from $1.56\cdot10^{-2}$ at $n=8$ to $1.84\cdot10^{-2}$ at $n=10$.
Thus the late increase is not confined to boundary nodes, even though the
global maximum is attained on the boundary in this run.
The direct solver in Section~\ref{sec:num_2d} reaches roundoff on the
corresponding problem with prescribed exact Dirichlet values.

To examine the influence of boundary supervision, we additionally freeze the
policy at $a_h^\ast$ and compare residual-only fitting with fitting that adds
the mean squared error against exact Dirichlet data on the four edges, with
weight $1$. This is a soft penalty, not exact enforcement.
Both diagnostic fits start from the same seed and use $8000$ Adam steps and
at most $300$ L-BFGS iterations. Relative to the outer-loop evaluator, Adam
weight decay is disabled, and L-BFGS uses history size $100$ and tolerances
$10^{-16}$. Hence this is not solely a fourfold-budget change.
The boundary-free and penalized fits respectively give grid maximum errors
$2.50\cdot10^{-2}$ and $3.77\cdot10^{-3}$, with interior residual maxima
$8.89\cdot10^{-3}$ and $1.04\cdot10^{-2}$.
These comparable, but unequal, residuals and different value errors support
the relevance of boundary control in \eqref{eq:inexact_pi}. A single pair of
fits does not establish an optimization-independent error floor.

\begin{figure}[!t]
  \centering
  \begin{subfigure}[t]{0.32\textwidth}
    \centering
    \includegraphics[width=\textwidth]{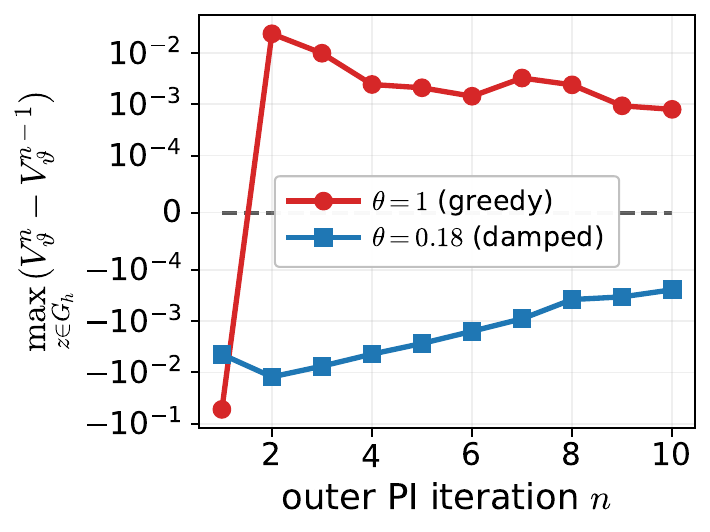}
    \caption{Signed maximum increment over all grid nodes for greedy and damped
    updates (symmetric log scale). Positive values indicate an increase.}
    \label{fig:pinn2d_mono}
  \end{subfigure}\hfill
  \begin{subfigure}[t]{0.32\textwidth}
    \centering
    \includegraphics[width=\textwidth]{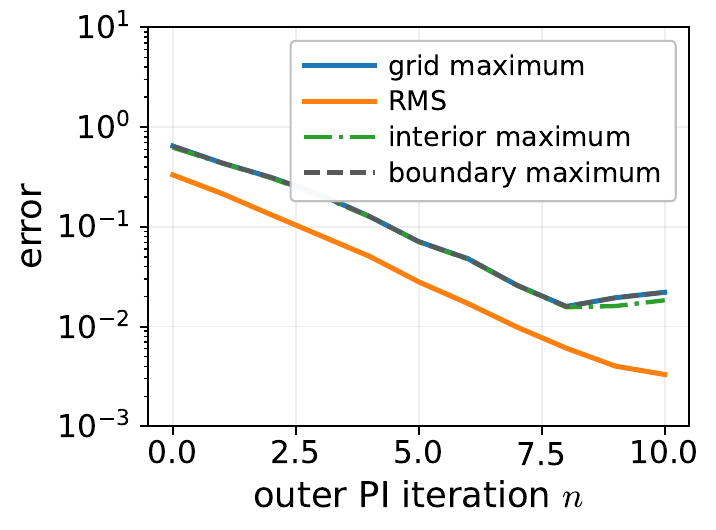}
    \caption{Grid maximum and RMS errors, with the maximum errors restricted
    to the interior and boundary, for $\theta=0.18$.}
    \label{fig:pinn2d_error}
  \end{subfigure}\hfill
  \begin{subfigure}[t]{0.32\textwidth}
    \centering
    \includegraphics[width=\textwidth]{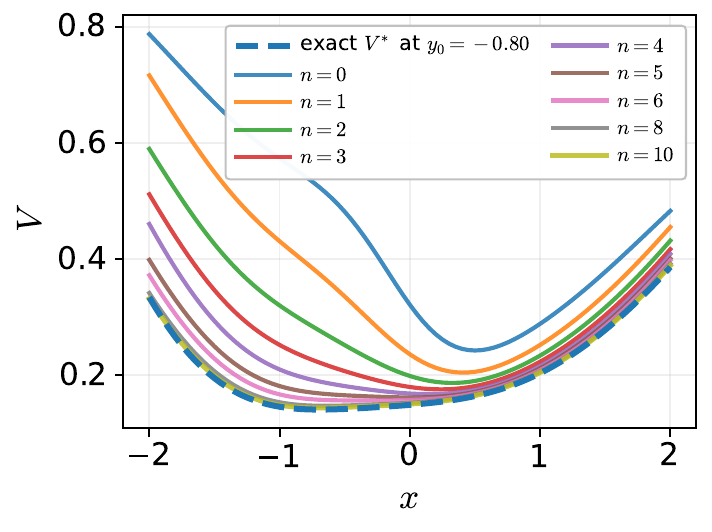}
    \caption{Iterates $V_\vartheta^n(\cdot,y_0)$ and $V^\ast(\cdot,y_0)$ at
    $y_0=-0.80$, for $\theta=0.18$.}
    \label{fig:pinn2d_slice_y}
  \end{subfigure}
  \caption{Boundary-free neural residual fitting with the same interior
  coefficients as Section~\ref{sec:num_2d}. Both the evaluator and boundary
  treatment differ from the Dirichlet finite-difference experiment. The left
  panel compares two updates; the other panels show the damped run.}
  \label{fig:pinn2d_main}
\end{figure}

\begin{remark}
These bounded-domain, boundary-free fits lie outside the whole-space
semi-discrete framework analyzed here. In particular, interior residual
control alone supplies no bound on the boundary error in \eqref{eq:inexact_pi}.
An analysis of approximate policy iteration with controlled residual and
boundary errors remains a direction for future work.
\end{remark}
}

\section{Conclusion}
{
We developed a semi-discrete policy iteration framework for deterministic
infinite-horizon discounted control with bounded Lipschitz data. Centered
differences permit pointwise policy improvement for nonsmooth values, the
artificial-viscosity term makes the stencil monotone, and the positive
discount gives the resolvent contraction. Under the regular policy-map
assumption, exact evaluation produces a decreasing sequence converging
geometrically to the unique semi-discrete solution. The local quadratic
estimate has a constant of order $h^{-2}$.

Under the additional condition $\lambda>\Lip_x(f)$, a doubling-of-variables
argument gives $\|V^h-V\|_\infty\le C\sqrt h$ using only the Lipschitz
regularity of the continuous value function. The bounded cusp example in
Remark~\ref{rem:sqrt_h_sharpness} establishes sharpness of this exponent.
Combining the error bounds gives a sufficient iteration count of order
$h^{-1}\log(1/h)$ to attain total error of order $\sqrt h$.

The bounded-domain experiments illustrate distinct aspects of the estimates.
The smooth 1D benchmark has a discretization plateau of order $h$, with the
stored fixed-mesh error agreeing with $NPh/\lambda$ to eight significant
digits. The 2D manufactured benchmark isolates convergence to the discrete
solution and reaches roundoff under exact evaluation. Both show much faster
local PI convergence than the guaranteed geometric rate. Over the tested
meshes, fixed-tolerance iteration counts remain below the sufficient
$O(h^{-1})$ bound; these observations do not establish an asymptotic law for
the actual counts.

The neural diagnostic fits the interior residual without prescribed boundary
values. Its greedy run exhibits increases in the values, while the damped
run does not over the reported updates. Because both the evaluator and
boundary treatment differ from the direct-solver experiment, this comparison
does not isolate the source of the loss of monotonicity. The frozen-policy
boundary-penalty comparison illustrates why approximate evaluation must
account for boundary errors as well as interior residuals.

A convergence analysis for approximate PI with controlled evaluation errors
remains open in this framework. Other directions include the undiscounted
case, where the strict contraction used here is lost, and scalable
approximations for high-dimensional problems.
}

\bibliographystyle{plain}
\bibliography{reference}

@book{Howard1960,
  title     = {Dynamic Programming and {M}arkov Processes},
  author    = {Howard, Ronald A.},
  year      = {1960},
  publisher = {The Technology Press of M.I.T. and John Wiley \& Sons},
  address   = {New York}
}

@book{Puterman1994,
  title     = {{M}arkov Decision Processes: Discrete Stochastic Dynamic Programming},
  author    = {Puterman, Martin L.},
  series    = {Wiley Series in Probability and Statistics},
  year      = {1994},
  publisher = {John Wiley \& Sons},
  address   = {New York},
  doi       = {10.1002/9780470316887}
}

@article{PutermanBrumelle1979,
  title     = {On the convergence of policy iteration in stationary dynamic programming},
  author    = {Puterman, Martin L. and Brumelle, Shelby L.},
  journal   = {Mathematics of Operations Research},
  volume    = {4},
  number    = {1},
  pages     = {60--69},
  year      = {1979},
  publisher = {INFORMS}
}

@article{Puterman1981,
  title     = {On the convergence of policy iteration for controlled diffusions},
  author    = {Puterman, Martin L.},
  journal   = {Journal of Optimization Theory and Applications},
  volume    = {33},
  number    = {1},
  pages     = {137--144},
  year      = {1981},
  publisher = {Springer},
  doi       = {10.1007/BF00935182}
}

@article{SantosRust2004,
  title     = {Convergence properties of policy iteration},
  author    = {Santos, Manuel S. and Rust, John},
  journal   = {SIAM Journal on Control and Optimization},
  volume    = {42},
  number    = {6},
  pages     = {2094--2115},
  year      = {2004},
  publisher = {SIAM},
  doi       = {10.1137/S0363012902399824}
}

@article{Kleinman1968,
  title     = {On an iterative technique for {R}iccati equation computations},
  author    = {Kleinman, David},
  journal   = {IEEE Transactions on Automatic Control},
  volume    = {13},
  number    = {1},
  pages     = {114--115},
  year      = {1968},
  publisher = {IEEE}
}

@article{Vrabie2009,
  title     = {Adaptive optimal control for continuous-time linear systems based on policy iteration},
  author    = {Vrabie, Draguna and Pastravanu, Octavian and Abu-Khalaf, Murad and Lewis, Frank L.},
  journal   = {Automatica},
  volume    = {45},
  number    = {2},
  pages     = {477--484},
  year      = {2009},
  publisher = {Elsevier},
  doi       = {10.1016/j.automatica.2008.08.017}
}

@article{KunduKunisch2024,
  title   = {Policy iteration for {H}amilton--{J}acobi--{B}ellman equations with control constraints},
  author  = {Kundu, Sudeep and Kunisch, Karl},
  journal = {Computational Optimization and Applications},
  volume  = {87},
  pages   = {785--809},
  year    = {2024},
  doi     = {10.1007/s10589-021-00278-3}
}

@article{BokanowskiMarosoZidani2009,
  title     = {Some convergence results for {H}oward's algorithm},
  author    = {Bokanowski, Olivier and Maroso, Stefania and Zidani, Hasnaa},
  journal   = {SIAM Journal on Numerical Analysis},
  volume    = {47},
  number    = {4},
  pages     = {3001--3026},
  year      = {2009},
  publisher = {SIAM}
}

@article{Kerimkulov2020,
  title     = {Exponential convergence and stability of {H}oward's policy improvement algorithm for controlled diffusions},
  author    = {Kerimkulov, Bekzhan and {\v S}i{\v s}ka, David and Szpruch, Lukasz},
  journal   = {SIAM Journal on Control and Optimization},
  volume    = {58},
  number    = {3},
  pages     = {1314--1340},
  year      = {2020},
  publisher = {SIAM},
  doi       = {10.1137/19M1236758}
}

@article{CrandallLions1984,
  title   = {Two approximations of solutions of {H}amilton--{J}acobi equations},
  author  = {Crandall, Michael G. and Lions, Pierre-Louis},
  journal = {Mathematics of Computation},
  volume  = {43},
  number  = {167},
  pages   = {1--19},
  year    = {1984}
}

@article{CrandallIshiiLions1992,
  title   = {User's guide to viscosity solutions of second order partial differential equations},
  author  = {Crandall, Michael G. and Ishii, Hitoshi and Lions, Pierre-Louis},
  journal = {Bulletin of the American Mathematical Society},
  volume  = {27},
  number  = {1},
  pages   = {1--67},
  year    = {1992},
  doi     = {10.1090/S0273-0979-1992-00266-5}
}

@article{BarlesSouganidis1991,
  title     = {Convergence of approximation schemes for fully nonlinear second order equations},
  author    = {Barles, Guy and Souganidis, Panagiotis E.},
  journal   = {Asymptotic Analysis},
  volume    = {4},
  number    = {3},
  pages     = {271--283},
  year      = {1991},
  publisher = {IOS Press}
}

@book{FlemingSoner2006,
  title     = {Controlled {M}arkov Processes and Viscosity Solutions},
  author    = {Fleming, Wendell H. and Soner, H. Mete},
  series    = {Stochastic Modelling and Applied Probability},
  volume    = {25},
  edition   = {2},
  year      = {2006},
  publisher = {Springer},
  address   = {New York}
}

@book{BardiCapuzzo1997,
  title     = {Optimal Control and Viscosity Solutions of {H}amilton--{J}acobi--{B}ellman Equations},
  author    = {Bardi, Martino and Capuzzo-Dolcetta, Italo},
  series    = {Systems \& Control: Foundations \& Applications},
  publisher = {Birkh\"auser},
  address   = {Boston},
  year      = {1997}
}

@book{tran2021hamilton,
  title     = {{H}amilton--{J}acobi Equations: Theory and Applications},
  author    = {Tran, Hung Vinh},
  series    = {Graduate Studies in Mathematics},
  volume    = {213},
  year      = {2021},
  publisher = {American Mathematical Society}
}

@article{TangTranZhang2025,
  title     = {Policy iteration for the deterministic control problems---a viscosity approach},
  author    = {Tang, Wenpin and Tran, Hung Vinh and Zhang, Yuming Paul},
  journal   = {SIAM Journal on Control and Optimization},
  volume    = {63},
  number    = {1},
  pages     = {375--401},
  year      = {2025},
  publisher = {SIAM},
  doi       = {10.1137/24M1631602}
}

@article{HuangWangZhou2025,
  title     = {Convergence of policy iteration for entropy-regularized stochastic control problems},
  author    = {Huang, Yu-Jui and Wang, Zhenhua and Zhou, Zhou},
  journal   = {SIAM Journal on Control and Optimization},
  volume    = {63},
  number    = {2},
  pages     = {752--777},
  year      = {2025},
  publisher = {SIAM},
  doi       = {10.1137/24M1638744}
}

@article{TranWangZhang2025,
  title     = {Policy iteration for exploratory {H}amilton--{J}acobi--{B}ellman equations},
  author    = {Tran, Hung Vinh and Wang, Zhenhua and Zhang, Yuming Paul},
  journal   = {Applied Mathematics and Optimization},
  volume    = {91},
  number    = {2},
  pages     = {50},
  year      = {2025},
  publisher = {Springer},
  doi       = {10.1007/s00245-025-10249-3}
}

@article{raissi2019physics,
  title     = {Physics-informed neural networks: a deep learning framework for solving
               forward and inverse problems involving nonlinear partial differential equations},
  author    = {Raissi, Maziar and Perdikaris, Paris and Karniadakis, George E.},
  journal   = {Journal of Computational Physics},
  volume    = {378},
  pages     = {686--707},
  year      = {2019},
  publisher = {Elsevier}
}

@article{han2018solving,
  title     = {Solving high-dimensional partial differential equations using deep learning},
  author    = {Han, Jiequn and Jentzen, Arnulf and E, Weinan},
  journal   = {Proceedings of the National Academy of Sciences},
  volume    = {115},
  number    = {34},
  pages     = {8505--8510},
  year      = {2018},
  publisher = {National Academy of Sciences}
}

@article{lee2025hamilton,
  title     = {{H}amilton--{J}acobi based policy-iteration via deep operator learning},
  author    = {Lee, Jae Yong and Kim, Yeoneung},
  journal   = {Neurocomputing},
  volume    = {646},
  pages     = {130515},
  year      = {2025},
  publisher = {Elsevier},
  doi       = {10.1016/j.neucom.2025.130515}
}

@article{kim2025neural,
  title   = {Neural policy iteration for stochastic optimal control: a physics-informed approach},
  author  = {Kim, Yeongjong and Kim, Yeoneung and Kim, Minseok and Cho, Namkyeong},
  journal      = {arXiv preprint arXiv:2508.01718},
  year         = {2025},
  eprint       = {2508.01718},
  archiveprefix = {arXiv},
  primaryclass = {math.OC}
}

@inproceedings{kim2026physics,
  title     = {Physics-informed approach for exploratory {H}amilton--{J}acobi--{B}ellman equations via policy iterations},
  author    = {Kim, Yeongjong and Cho, Namkyeong and Kim, Minseok and Kim, Yeoneung},
  booktitle = {Proceedings of the AAAI Conference on Artificial Intelligence},
  volume    = {40},
  pages     = {22609--22616},
  year      = {2026},
  note      = {Issue 27; arXiv:2508.01720}
}

@book{SuttonBarto2018,
  title     = {Reinforcement Learning: An Introduction},
  author    = {Sutton, Richard S. and Barto, Andrew G.},
  edition   = {2},
  year      = {2018},
  publisher = {MIT Press},
  address   = {Cambridge, MA}
}

\end{document}